\documentclass[journal]{IEEEtran}
\usepackage{CJK}
\usepackage{amssymb}
\usepackage{epstopdf}
\usepackage{graphicx}
\usepackage{multirow}
\usepackage{diagbox}
\usepackage{algorithmicx,algorithm}
\usepackage{framed}
\usepackage{booktabs}
\usepackage{threeparttable}
\usepackage{comment}
\usepackage{amsmath}
\usepackage{amsthm}
\usepackage{xcolor}
\graphicspath{ {./images/} }

\newtheorem{thm}{Theorem}
\newtheorem{lem}{Lemma}

\newtheorem{pro}{Proposition}

\newtheorem{defi}{Definition}
\newtheorem{ass}{Assumption}

\newcommand {\emptycomment}[1]{}

\newcommand{\be }{\begin{equation}}
\newcommand{\ee }{\end{equation}}

\newcommand{\noi}{\noindent}

\newcommand{\wh}{\widetilde}

\newcommand{\f}{\frac}

\newcommand{\nn}{\langle}
\newcommand{\mm}{\rangle}

\newcommand{\huaX}{\mathcal{X}}

\newcommand{\mbb}{\mathbb}

\newcommand{\ooo}{\omega}

\def\bea{\begin{eqnarray}}
\def\eea{\end{eqnarray}}
\def\be{\begin{equation}}
\def\ee{\end{equation}}
\def\blm{\begin{lem}}
\def\elm{\end{lem}}
\def\btm{\begin{theorem}}
\def\etm{\end{theorem}}

\newcommand{\huaB}{\mathcal{B}}

\newcommand{\huaL}{\mathcal{L}}
\newcommand{\huaR}{\mathcal{R}}
\newcommand{\huaE}{\mathcal{E}}

\newcommand{\huaG}{\mathcal{G}}

\newcommand{\huaU}{\mathcal{U}}
\newcommand{\huaV}{\mathcal{V}}
\newcommand{\huaW}{\mathcal{W}}
\newcommand{\huaY}{\mathcal{Y}}

\newcommand{\huaP}{\mathcal{P}}
\newcommand{\huaC}{{\mathcal{C}}}

\newcommand{\huaH}{\mathcal{H}}

\newcommand{\huaO}{\mathcal{O}}
\newcommand{\huaT}{\mathcal{T}}

\newcommand{\huaM}{\mathcal{M}}

\def\bea{\begin{eqnarray}}
\def\eea{\end{eqnarray}}
\def\be{\begin{equation}}
\def\ee{\end{equation}}
\def\blm{\begin{lem}}
\def\elm{\end{lem}}

\def\bea{\begin{eqnarray}}
	\def\eea{\end{eqnarray}}
\def\be{\begin{equation}}
	\def\ee{\end{equation}}
\def\blm{\begin{lem}}
	\def\elm{\end{lem}}
\def\bes{\begin{eqnarray*}}
	\def\ees{\end{eqnarray*}}
\def\beal{\begin{aligned}}
	\def\eeal{\end{aligned}}

\def\wh{\widehat}

\def\ww{\widetilde}
\def\sss{\sigma}
\def\tb{\textbf}
\def\mb{\mathbf}

\newtheorem{theorem}{Theorem}

\begin{document}

\title{Generic Frameworks for Distributed Functional Optimization and Learning over Time-Varying Networks}
\author{Zhan Yu, Zhongjie Shi, Deming Yuan, \IEEEmembership{Senior Member, IEEE}, Daniel W. C. Ho, \IEEEmembership{Life Fellow, IEEE}\thanks{
		Z. Yu is with the Department of Mathematics, Hong Kong Baptist University. Z. Shi is with the School of Mathematics, Georgia Institute of Technology, Atlanta, GA 30332. D. W. C. Ho is with the Department of Mathematics, City University of Kong Kong. D. Yuan is with the School of Automation, Nanjing University of Science and Technology (dmyuan1012@gmail.com). (\emph{Corresponding Author: Deming Yuan})}}
\maketitle


\begin{abstract}
In this paper, we establish a distributed functional optimization (DFO) theory over time-varying  networks. The vast majority of existing distributed optimization theories are developed based on Euclidean decision variables. However, for many scenarios in machine learning and statistical learning, such as reproducing kernel spaces or probability measure spaces that use functions or probability measures as fundamental variables, the development of existing distributed optimization theories exhibit  obvious theoretical and technical deficiencies. This paper addresses these issues by developing a novel general DFO theory on Banach spaces, allowing functional learning problems in the aforementioned scenarios to be incorporated into our framework for resolution. We study both convex and nonconvex DFO problems and rigorously establish a comprehensive convergence theory of distributed functional mirror descent and distributed functional gradient descent algorithm to solve them.  Satisfactory convergence rates are fully derived. The work has  provided  generic analyzing frameworks for DFO. The established theory is shown to have crucial application value in the kernel-based distributed learning theory over networks.

\end{abstract}

\section{Introduction}
 Since Nedi\'c and Ozdaglar's foundational work \cite{no2009}, multi-agent distributed  optimization (MA-DO) has played a crucial role in  multiple fields, including modern control theory, optimization theory, and networked systems in the past fifteen years \cite{no2009}-\cite{yhl2016}. MA-DO utilizes a multi-agent network structure, endowing the entire algorithmic framework with a highly flexible node information exchange and collaboration mechanisms that traditional centralized algorithms lack. Moreover, in contrast to classical centralized methods, MA-DO offers greater scalability and resilience, making them more suitable for dynamic and uncertain environments. Meanwhile, MA-DO also plays an important role in privacy protection in data-based distributed supervised learning problems. These benefits position DO as an essential basis for developing effective algorithms in modern control and optimization society.

 Despite the revolutionary advancements and progress achieved by MA-DO, we observe a fundamental fact: the vast majority of distributed optimization algorithm theories are constructed based on Euclidean decision spaces \cite{no2009}-\cite{yhl2016}. Accordingly, their related decision variables/state variables are obviously Euclidean vectors. However, numerous examples in machine learning, statistical learning, and neural-network-based deep learning indicate that the learning theory focused on approximating functions defined in some typical function spaces that are essentially infinite-dimensional (such as reproducing kernel Hilbert space (RKHS) $H_K$, Lebesgue space $L^p$, Sobolev space $W_p^s$, Besov space $B_{p,q}^s$ for some integrability index $p$, $q$ and regularity index $s$) have gradually become dominant in these fields (e.g. \cite{yfsz2024}-\cite{shs2001}). For example, in kernel-based learning theory, in an RKHS ($H_K$, $\|\cdot\|_{H_K}$) induced by a Mercer kernel $K$, the least squares regression problem over RKHS often requires us to obtain a functional estimator by minimizing a data-based nonlinear convex functional 
 in an RKHS ball 
 or the whole RKHS (\cite{yfsz2024}, \cite{kbp2024},  \cite{ghs2018}, \cite{ghw2020}, \cite{gsw2017}). 
 In such models, the learning target is a functional defined on infinite-dimensional function space, in significant contrast to the classical settings that MA-DO handles. This reality also imposes significant bottlenecks on the applicability of traditional DO algorithms in functional learning and variational problems. Current  MA-DO methods for these scenarios are still quite limited. As these fields increasingly rely on complex function spaces, the constraints of classical MA-DO methods hinder their effectiveness and adaptability, necessitating the development of more appropriate approaches that can accommodate these advanced learning frameworks. 

 In this work, with the aid of time-varying multi-agent systems, we will investigate a new distributed functional optimization and learning problem framework over general Banach spaces. Many problem instances in modern machine learning and statistical learning can be categorized within our  multi-agent distributed functional optimization (MA-DFO) framework. Here, we consider the MA-DFO problem, which aims at minimizing a global functional $\mb J$ defined on some Banach spaces $\huaB$ by utilizing a multi-agent system/network modeled by a time-varying graph with a node set indexed by $i=1,2,...,m$.
 Our problem formulation can be described as a distributed functional optimization (DFO) model 
 $\min_{f\in\huaW}\tb J(f):=\sum_{i=1}^m\tb J_i(f)$ 
 where $\huaW$ is a closed convex set of Banach space $\huaB$. Here, for some practical considerations such as privacy protection,  we handle the case that the functional $\mb J_i$ is known only to the local agent $i$ of the networked system. It should be noted that the decision variables here can be basic Euclidean vectors, a situation that has been extensively studied in existing work \cite{no2009}-\cite{yhl2016}. Our work focuses on the case where $f$ belongs to the typical function spaces or measure spaces mentioned above. For this case, the DFO problem framework is sensible since in real-world applications, there are many scenarios that data are stored  across an interconnected network, such as the  robotic team systems and sensor networks (see e.g.  \cite{ks2004}). In such setting, each agent $i$ has only access to a local data set $D_i=\{(x_{i,s},y_{i,s})\}_{s=1}^{n_i}$ and collaboratively learn a global regression function with a general loss functional $\huaL$ via the supervised learning framework
 $f^*=\arg\min_{\|f\|_\huaB\leq R}\sum_{i=1}^m\f{1}{n_i}\sum_{s=1}^{n_i}\huaL(f(x_{i,s}),y_{i,s})$.
Here, the decision domain $\{f\in\huaB:\|f\|_\huaB\leq R\}$ can be an RKHS ball, Sobolev ball or Besov ball for different machine learning tasks. This model can naturally be treated as a special case of our general DFO problem framework. The idea of each agent managing a subset of data has been around for some time, particularly in federated learning and distributed kernel-based learning algorithms that employ a divide-and-conquer approach (see e.g., \cite{yfsz2024}, \cite{ghw2020}, \cite{gsw2017}). Due to privacy concerns, the global dataset often must be partitioned into disjoint sub-datasets for processing by multiple agents. Additionally, there are situations where data are inherently stored across multiple local agents in a distributed manner, without being combined at the outset due to concerns about privacy and the need to minimize potential costs, as widely observed in distributed/federated learning, financial markets, and medical systems.

This work will comprehensively address both   convex and nonconvex DFO. For convex DFO on Banach spaces, inspired by the remarkable success achieved by mirror descent approaches to handling different non-Euclidean scenarios (see e.g. \cite{yhhx2020}-\cite{xzhyx2022}, \cite{kbp2024}, \cite{akl2022}, \cite{lz2018}, \cite{hljl2025}), we will employ mirror descent structure  and propose a distributed functional mirror descent (DFMD) algorithm in the framework of general Banach space $\huaB$ to solve the DFO problem. Additionally, noting that in many learning tasks, to address specific learning problems, the loss function often appears in a nonconvex form. Typical examples include, for example, Sigmoid loss, truncated quadratic loss in classification problems and Cauchy loss, Welsch loss in robust learning problems and  correntropy-induced loss, and minimum-error-entropy loss in  information-theoretic
learning (see e.g.  \cite{ghs2018}, \cite{ghw2020}).  These nonconvex loss functionals play a crucial role in the corresponding learning scenarios. However, effective existing multi-agent distributed optimization theories specifically for such nonconvex loss functionals are clearly lacking. This makes the development of appropriate multi-agent nonconvex DFO highly significant.
 We will explore the nonconvex DFO problem in the framework of general Hilbert space $\huaH$ by introducing a distributed functional gradient descent (DFGD) method.

We summarize the  contributions of this paper as follows.
\begin{itemize}
	\item We rigorously formulate a novel time-varying multi-agent DFO framework on general Banach spaces that are essentially infinite-dimensional. To solve DFO for convex functional, we establish the convergence theory of DFMD and shows that it is able to achieve an optimal ergodic convergence rate of $\huaO(1/\sqrt{T})$. Moreover, it is worth mentioning that, in our analytical framework, even for the most general case of Banach spaces, obtaining this convergence rate does not require any  boundedness assumptions for the decision spaces, which are necessary theoretical prerequisites in the vast majority of existing work related to mirror descent approaches, even for Euclidean optimization, including \cite{yhhx2020}-\cite{xzhyx2022}.
	
	\item To our knowledge, a time-varying multi-agent distributed functional nonconvex optimization framework over Hilbert spaces is first proposed. The convergence theory of DFGD is comprehensively established. We show that under appropriate stepsize rules, a  convergence rate of $\huaO(1/\sqrt{T})$ for the local ergodic sequence  can be derived under mild conditions. Moreover, for the local sequence of the last iteration associated with any agent, we further obtain an $R$-linear convergence rate up to a solution level that is proportional to stepsize.
	
	\item As crucial byproducts of the theory developed in this work, we show that the kernel-based learning theory on the minimization problem of the expected risk of prediction function in the RKHS framework and the variational problems over Radon measure spaces can be well transformed to be handled in our time-varying MA-DFO framework. As a result, we have established generic DFO algorithmic and theoretical frameworks that are widely applicable to various scenarios in control, machine learning, and statistical learning. The research also provides a starting point for studying DFO theory using time-varying multi-agent networks, triggering one of the future directions of DO on Banach spaces and related non-Euclidean spaces.
		\end{itemize}

\noi \textbf{Notation:} For a Banach  space ($\huaB$, $\|\cdot\|_\huaB$). We use ($\huaB^*$, $\|\cdot\|_{\huaB^*}$) to denote its dual space, namely, the Banach space of bounded/continuous linear functionals on $\huaB$. Accordingly, for a functional $\mb F:\huaB\rightarrow\mbb R$, $\|\mb F\|_{\huaB^*}=\sup_{\|f\|_{\huaB}\leq1}\mb F(f)$, $f\in\huaB$.
For a matrix $M\in\mathbb R^{m\times m}$, denote the element in $i$th row and $j$th column by $[M]_{ij}$. For an $n$-dimension Euclidean vector $x$, denote its $i$-th component by $[x]_i$, $i=1,2,...,n$.  

\section{Problem setting}
In this paper, we aim at solving the DFO problem 
\bea
\min_{f\in\huaW}\tb J(f):=\sum_{i=1}^m\tb J_i(f)  \label{problem_formulation}
\eea
over Banach spaces, where $\huaW$ is a closed convex set of a real Banach space $\huaB$ that can be unbounded in the sense of the metric induced by the Banach space norm.  This paper addresses this DFO problem by utilizing a time-varying multi-agent network. The agents of the multi-agent network are indexed by  $i=1,2,..., m$. In our model, the functional $\mb J_i$ is known only to one local agent of the multi-agent system. The communication topology among the agents is modeled as a time-varying graph $\huaG_t=(\huaV,\huaE_t, P_t)$, where $\huaV=\{1,2,..., m\}$ is the node set of the multi-agent system,  $\huaE_t$ is
the set of activated links (the edge set) at time $t$. $P_t$ is the communication matrix associated with the graph $\huaG_t$ such that $[P_t]_{ij}>0$ if $(j,i)\in\huaE_t$ and $[P_t]_{ij}=0$ otherwise. Accordingly, $\huaE_t$ can be described as $\huaE_t=\{(j,i)\in\huaV\times\huaV|[P_t]_{ij}>0\}$. We refer to $\{j\in\huaV|(j,i)\in\huaE_t\}$ as the neighbor set of agent $i\in\huaV$ at time instant $t$. We require these local agents collaboratively to solve the DFO problem via appropriate local communications over Banach spaces. We suppose that there is at least an optimal element $f^*\in\huaB$ such that $\mb J(f)\geq \mb J(f^*)$ for any $f\in\huaB$.

Now, we describe the network environment. In this paper, we establish our theory based on the following class of graphs.
\begin{ass}\label{graph_ass} The communication matrix $P_t$ is a doubly stochastic matrix, $i.e.$, $\sum_{i=1}^m[P_t]_{ij}=1$ and $\sum_{j=1}^m[P_t]_{ij}=1$ for any $i$ and $j$. There exists an integer $B\geq1$ such that the graph $(\huaV, \huaE_{kB+1}\cup\cdots\cup\huaE_{(k+1)B})$ is strongly connected for any $k\geq0$. There exists a scalar $0<\zeta<1$ such that $[P_t]_{ii}\geq \zeta$ for all $i$ and $t$, and $[P_t]_{ij}\geq\zeta$ if $(j,i)\in \huaE_t$.
\end{ass}


Let $\tb F:\huaB\rightarrow\mbb R$ be a functional defined on a closed convex set $\huaW$ of Banach space ($\huaB$, $\|\cdot\|_\huaB$) (a complete normed vector space), if for any $f,g\in\huaW$ and $\lambda\in[0,1]$, there holds, $\tb J(\lambda f+(1-\lambda)g)\leq\lambda\tb J(f)+(1-\lambda)\tb J(g)$, then we call $\tb J$ a convex functional on $\huaW\subseteq\huaB$. For the Banach space $\huaB$, this work can cover the most typical function spaces for the scenarios of learning theory, statistical learning, and deep learning, such as reproducing kernel Hilbert space ($H_K$,$\|\cdot\|_{H_K}$) induced by a Mercer $K$, the continuous function space ($C(\huaX)$,$\|\cdot\|_{C(\huaX)}$) defined on some compact metric space $\huaX$, the well-known fundamental ($L^p$, $\|\cdot\|_{L^p}$) Lebesgue function space or Sobolev space $(W_p^s,\|\cdot\|_{W_p^s})$. It is worth noting that in the existing DO frameworks based on multi-agent systems, the theoretical work on algorithms established based on these typical Banach spaces is still quite scarce.

Now, we start to present our theoretical framework. First, we  provide the definition of the functional derivative. The following definitions give the concepts of G\^ateaux differentiability and Fr\'echet differentiability (\cite{luenberger1997}, \cite{Zalinescu2002}, \cite{Dieudonne_book}).
\begin{defi}\label{gateaux_def}
	Let $\huaX$ and $\huaY$ be  locally convex topological vector spaces (e.g., Banach spaces). $\huaU\subseteq\huaX$ is open, $\mb F:\huaU\rightarrow \huaY$ is a nonlinear mapping. Let $f\in\huaU\subseteq\huaB$ and $h\in\huaB$ be arbitrary elements in $\huaB$. The G\^ateaux differential $\partial_f\mb F(h)$ of $\mb F$ at $f\in\huaU$ in the direction $h\in\huaX$ is defined as
	$\partial_f\mb F(h)=\lim_{\tau\rightarrow0}\f{\mb F(f+\tau h)-\mb F(f)}{\tau}$.
	If the limit exists for all $h$, then one says that $\mb F$ is G\^ateaux differentiable at $f$.
\end{defi}

\begin{defi}\label{frechet_def}
	Let $\huaX$ and $\huaY$ be normed vector spaces. Let $\huaU\subseteq\huaX$ be an open set of $\huaX$. A mapping $\mb F:\huaU\rightarrow\huaY$ is called Fr\'echet differentiable at $f\in\huaU$ if there exists a bounded linear operator $\mb A:\huaX\rightarrow\huaY$ such that
	$\lim_{\|h\|_\huaX\rightarrow0}\f{\|\mb F(f+h)-\mb F(f)-\mb A h\|_\huaY}{\|h\|_\huaX}=0$.
	Such operator $\mb A$ is unique, accordingly we denote $\mb D_f \mb F=\mb A$, and call $\mb D_f\mb F$ the Fr\'echet derivative of $\mb F$ at $f$.
\end{defi}

We recall that, according to the definition of the dual space $\huaB^*$ of $\huaB$, an element $\mb F\in\huaB^*$ induces a natural action on any element of $\huaB$. In the subsequent analysis of this paper, we denote such action via a bilinear operator induced by the dual bracket $\nn\cdot,\cdot\mm_{\huaB\times\huaB^*}:\huaB\times \huaB^*\rightarrow\mbb R$. Namely, 
$\mb F(f)=\nn f,\mb F\mm_{\huaB\times\huaB^*}$.
Next, we give the notion of a subgradient set of a functional at certain element of Banach space.
\begin{defi}
	For a functional $\mb F:\huaB\rightarrow \mbb R$, its set of subgradients at any $f_0\in\huaB$ is defined as 
	$\partial_{f_0}\mb F=\{g\in\huaB^*:\forall f\in\huaB, \mb F(f)-\mb F(f_0)\geq\nn f-f_0,g\mm_{\huaB\times\huaB^*}\}$.
\end{defi}

It is worth mentioning that, for a convex functional $\mb F$, if it exists G\^ateaux differential at some element $f\in\huaB$, then the subgradient set has a unique element $\partial_f \mb F$, namely the G\^ateaux differential of the functional $\mb F$.  Throughout this paper, we will continue to use the notion $\partial_f \mb F$ to represent a randomly selected subgradient $g\in\partial_f \mb F$. 

The following definition describes the notion of Lipschitz condition for a convex functional.
\begin{defi}
	A convex functional $\tb F:\huaB\rightarrow\mbb R$ is $G$-Lipschitz on a subset $\huaU\subseteq\huaB$ with respect to $\|\cdot\|_\huaB$ if for any $f\in\huaU$ and subgradient $g\in\partial_f \tb F$, it holds that $\|g\|_{\huaB^*}\leq G$. 
\end{defi}

Now, we are ready to state the following assumption related to Lipschitz's condition of the local functionals.
\begin{ass}\label{gradient_ass}
	Local objective function $\mb J_i$, $i\in\huaV$ is $G$-Lipschitz in $\huaW\subseteq\huaB$, $i.e.$ $\left\|\partial_f\mb J_i\right\|_{\huaB^*}\leq G$, $f\in\huaW$.
\end{ass}

For developing the nonconvex DFO theory, we also require the following notions on $L$-smoothness of a functional.
\begin{defi}\label{Lsmooth_def}
	A functional $\mb F$ is $L$-smooth if for all $f_0\in\huaB$, there exists $g\in\partial_{f_0}\mb F$ such that 
	$\mb F(f)-\mb F(f_0)\leq \nn f-f_0,g\mm_{\huaB\times\huaB^*}+\f{L}{2}\|f-f_0\|_\huaB^2, \ \forall f\in\huaB$.
\end{defi}
From Corollary 3.5.7 in \cite{Zalinescu2002}, we know that the above inequality implies that, if a functional $\mb F$ is Fr\'echet differentiable and $L$-smooth, then we also have 
$\|\mb D_f\mb F-\mb D_{f'}\mb F\|_{\huaB^*}\leq L\|f-f'\|_\huaB$.

In this paper, without loss of generality, we consider the setting that $\huaB$ is a reflexive Banach space, namely, the natural embedding $\huaB\rightarrow\huaB^{**}$ is an isometric isomorphism ($\huaB\cong\huaB^{**}$), where $\huaB^{**}=(\huaB^*)^*$ is the bidual space of $\huaB$. Moreover,  we consider that the domain $\huaW\neq\emptyset$ and $\mb J_i(f)>-\infty$ for any $f\in\huaB$. Hence, $\mb J_i$, $i\in\huaV$ are proper. In the first part of this work, we consider the case that the functionals $\mb J_i$, $i\in\huaV$ are strictly convex on a closed convex set $\huaW\subseteq\huaB$ ($\huaW$ can be unbounded). In the second part, we consider the case that $\mb J_i$ are Fr\'echet differentiable and can be nonconvex in the setting that the underlying space is a Hilbert space $\huaB=\huaH$. These two scenarios can encompass many core problem formulations in the fields of modern control and machine learning, which have been presented earlier and will be illustrated with typical examples later (see Section \ref{section5}).

This paper proposes to establish a distributed function mirror descent (DFMD) framework for solving the convex DFO problem at first. To this end, we need to introduce the notion of  mirror map at the beginning.
\begin{defi}
	We call a functional $\mb \Psi$  mirror map if it satisfies: $\mb \Psi$ is strictly convex and the subgradient sets of $\partial_{(\cdot)}\mb \Psi$ does not intersect at non-empty set and $\partial_{(\cdot)}\mb \Psi(\huaB)=\huaB^*$. 
\end{defi}

 To avoid certain extreme and trivial cases, throughout the paper, we will select a proper and coercive mirror map $\mb \Psi$ that is G\^ateaux differentiable and $\sigma_{\mb \Psi}$-strongly convex. Here, by saying the mirror map $\mb \Psi$ is $\sigma_{\mb \Psi}$-strongly convex, we mean for all $f_0\in\huaB$, there exists $g\in\partial_{f_0}\mb \Psi$ such that 
 \bes
 \mb \Psi(f)-\mb \Psi(f_0)\geq \nn f-f_0,g\mm_{\huaB\times\huaB^*}+\f{\sigma_{\mb \Psi}}{2}\|f-f_0\|_\huaB^2, \ \forall f\in\huaB.
 \ees 
 In the existing mirror descent theory based on Euclidean decision variables, strongly convex mirror maps have been widely adopted \cite{yhhx2020}-\cite{xzhyx2022}. The following result is basic in terms of establishing our convergence theory (see e.g. \cite{kbp2024}).
\begin{lem}\label{bijective_lemma}
	Let $\mb F$ be a proper, strictly convex, and G\^ateaux differentiable functional. Then the operator $\partial_{(\cdot)}\mb F:\huaB\rightarrow\huaB^*$ is both injective and subjective. Here, for any $f\in\huaB$, $f\rightarrow\partial_{(\cdot)}\mb F(f):=\partial_f\mb F\in\huaB^*$.
\end{lem}
Now, we can provide the definition of the Bregman divergence.
\begin{defi}
	For a real-valued convex functional $\tb F:\huaB\rightarrow \mbb R$, the Bregman divergence $D_{\tb F}$ with respect to $\tb F$ is defined as
	\bes
	D_{\tb F}(f_1\|f_2):=\tb F(f_1)-\tb F(f_2)-\left\nn f_1-f_2,\partial _{f_2}\tb F\right\mm_{\huaB\times\huaB^*}
	\ees
	where $\partial_{(\cdot)}\tb F:\huaB\rightarrow\huaB^*$ denotes the sub-differential of the functional $\tb F$.
\end{defi}
The above definition indicates an obvious three-point identity. Namely, for any $f,g,h\in\huaB$, there holds
\bes
\left\nn f-h,\partial_f\mb F-\partial_g \mb F\right\mm_{\huaB\times\huaB^*}=D_{\mb F}(f\|g)+D_{\mb F}(h\|f)-D_{\mb F}(h\|g).
\ees
Now, we can give the notion of the projection map, this map is important to represent our main DFMD algorithm. For a mirror map $\mb \Psi$, the projection map $\Pi_{\huaW,\mb \Psi}$ is defined as 
\bes
\Pi_{\huaW,\mb \Psi}(f')=\arg\min_{f\in\huaW}D_{\mb \Psi}(f\|f').
\ees
Basic functional convex optimization theory indicates that, in a reflexive Banach space $\huaB$, a coercive, strongly convex functional has a unique global minimum on any closed convex set of $\huaB$ (e.g., \cite{luenberger1997}, \cite{Zalinescu2002}). This result ensures the meaningfulness of the above projection map and the rigorism of the theory that follows. The following is a basic assumption on separate convexity. Separate convexity is a widely considered assumption in the literature on DMD algorithms, as seen in works such as  \cite{yhhx2020}, \cite{md1}, \cite{lj2021}. The following assumption provides a generalized version of separate convexity on Banach spaces.
\begin{ass}\label{Bregman_convexity_ass}
	For any $f\in\huaB$, the Bregman divergence as a functional of the second variable $D_{\mb\Psi}(f\|\cdot)$ is assumed to satisfy the separate convexity on the second variable, namely, for any $\sum_{j=1}^ma_j=1$, $a_j\geq 0$, it holds that $D_{\mb \Psi}(f\| \sum_{j=1}^ma_jg_j)\leq\sum_{j=1}^ma_jD_{\mb \Psi}(f\|g_j)$, $g_j\in\huaB, j\in\huaV$.
\end{ass}

\section{Distributed functional convex optimization over general Banach spaces}\label{DFMD}

 In this section, we study the convex DFO problem, and $\mb J_i$, $i\in\huaV$ are supposed to be G\^ateaux differentiable convex functionals. Considering the tremendous power of the mirror descent framework in dealing with non-Euclidean scenarios, we plan to utilize a distributed functional mirror descent (DFMD) method to solve distributed functional convex optimization over general Banach spaces. Start with an initialization $f_{i,1}$, $i\in\huaV$, the DFMD is given as follows: 
\bea
\textbf{DFMD}\left\{
\begin{aligned}
	g_{i,t}=&\partial_{f_{i,t}}\mathbf \Psi,\\
	h_{i,t}=&g_{i,t}-\eta\partial_{f_{i,t}}\tb J_i,\\
	\wh f_{i,t+1}=&\Pi_{\huaW,\mathbf \Psi}\left((\partial \mathbf \Psi)^{-1}(h_{i,t})\right),\\
	f_{i,t+1}=&\sum_{j=1}^m[P_t]_{ij}\wh f_{j,t+1}.
\end{aligned}   \label{DFMD_algorithm}
\right. 
\eea
We briefly describe the mechanism of DFMD here. For each local agent $i\in\huaV$, after obtaining the initial $f_{i,1}$ through initialization. The algorithm DFMD utilizes the mirror map $\mb\Psi$  to map the initial value to the dual space $\huaB^*$ of the Banach space $\huaB$ through the operator $\partial_{(\cdot)}\mb \Psi$. Then, DFMD performs a gradient descent update with stepsize $\eta>0$ to obtain $h_{i,t}$ over the dual Banach space $\huaB^*$. In what follows, by utilizing the projection map $\Pi_{\huaW,\mb \Psi}$, $h_{i,t}$ is pulled back to the original Banach space, resulting in $\wh f_{i,t}\in\huaB$. Finally, these $\wh f_{i,t}$ communicate with their neighbors via the communication matrix $P_t$ to obtain $f_{i,t+1}$, completing an iteration step.

Before establishing our convergence theory, we emphasize that, for the sake of simplification, we consider the constant step size strategy with $\eta>0$ throughout the paper. For the case of a strictly decreasing stepsize strategy, the corresponding convergence results can be obtained by following a similar approach, and it might only require more computational steps, especially in the non-convex DFO section of this paper.

Now, we begin to establish the convergence theory of the DFMD algorithm. The following proposition is the foundation of the convergence result of DFMD.

\begin{pro}
	Under Assumption \ref{gradient_ass}, there holds that 
	\bes
	\left\|\wh f_{i,t+1}-f_{i,t}\right\|_\huaB\leq\f{G}{\sigma_{\mb\Psi}}\eta.
	\ees
\end{pro}
\begin{proof}
Due to the fact that the mirror map $\mb \Psi$ is proper, strongly convex and G\^ateaux differentiable, Lemma \ref{bijective_lemma} indicates that the operator $\partial_{(\cdot)}\mb \Psi:\huaB\rightarrow\huaB^*$ is invertible. If we denote $\mb \Psi^*:\huaB^*\rightarrow\huaB$ the convex conjugate of $\mb \Psi$, we have $\partial_{(\cdot)}\mb \Psi^*=(\partial_{(\cdot)}\mb \Psi)^{-1}$ as a result of the fact that $\huaB$ is reflexive. With the structure of the projection map $\Pi_{\huaW,\mb \Psi}$ at hand, according to the first-order optimality, we have
\bes
\left\nn  g-\wh f_{i,t+1}, \partial_{\wh f_{i,t+1}}D_{\mb \Psi}(\cdot,\partial_{h_{i,t}}\mb \Psi^*) \right\mm_{\huaB\times\huaB^*}\geq0, \  \forall g\in\huaW.
\ees
 For any $f'\in\huaW$, $D_{\mb \Psi}(\cdot,f')$ is a convex functional. Accordingly, for any $f_1$ and $f_2$, we have
\bes
\partial_{f_1}D_{\mb \Psi}(\cdot,f_2)=\partial_{f_1}\mb \Psi-\partial_{f_2}\mb \Psi.
\ees
Therefore, it holds that
\bes
\left\nn g-\wh f_{i,t+1},\partial_{\wh f_{i,t+1}}\mb \Psi-\partial_{\partial_{h_{i,t}}\mb \Psi^*}\mb \Psi\right\mm_{\huaB\times\huaB^*}\geq0, \  \forall g\in\huaW.
\ees
After setting $g=f_{i,t}$, we have
\bes
\left\nn \wh f_{i,t+1}-f_{i,t},\partial_{\wh f_{i,t+1}}\mb \Psi-\partial_{\partial_{h_{i,t}}\mb \Psi^*}\mb \Psi\right\mm_{\huaB\times\huaB^*}\leq0.
\ees
From the update of the main algorithm, we know
\bes
\partial_{\partial_{h_{i,t}}\mb \Psi^*}\mb \Psi=\partial_{f_{i,t}}\mb \Psi-\eta\partial_{f_{i,t}}\mb J_i.
\ees
Substituting this relation into the above inequality, we have
\bes
\left\nn \wh f_{i,t+1}-f_{i,t},\partial_{\wh f_{i,t+1}}\mb \Psi-\partial_{f_{i,t}}\mb \Psi+\eta\partial_{f_{i,t}}\mb J_i\right\mm_{\huaB\times\huaB^*}\leq0.
\ees
This further indicates that
\bes
&&\left\nn f_{i,t}-\wh f_{i,t+1},\eta\partial_{f_{i,t}}\mb J_i\right\mm_{\huaB\times\huaB^*}\\
&&\geq\left\nn \wh f_{i,t+1}-f_{i,t},\partial_{\wh f_{i,t+1}}\mb \Psi-\partial_{f_{i,t}}\mb \Psi\right\mm_{\huaB\times\huaB^*}\\
&&\geq\sigma_{\mb \Psi}\left\|\wh f_{i,t+1}-f_{i,t}\right\|_\huaB^2,
\ees
where we have used the equivalence relation of Corollary 3.5.11 in \cite{Zalinescu2002}. As a result of the basic property of bounded linear functional $\eta\partial_{f_{i,t}}\mb J_i$  defined on Banach space $\huaB$, we have
\bes
\left|\left\nn f_{i,t}-\wh f_{i,t+1},\eta\partial_{f_{i,t}}\mb J_i\right\mm_{\huaB\times\huaB^*}\right|\leq \eta\left\|\partial_{f_{i,t}}\mb J_i\right\|_{\huaB^*}\left\|\wh f_{i,t+1}-f_{i,t}\right\|_\huaB.
\ees
Finally, we arrive at
$\|p_{i,t}\|_\huaB\leq\f{G}{\sigma_{\mb\Psi}}\eta$,
which completes the proof.
\end{proof}

Before stating the next result, we require the following lemma related to the behavior of time-varying network. For any $t\geq s\geq0$, we denote the transition matrix $Q(t,s)$ associated with the communication matrix $P_t$ by $Q(t,s)=P_tP_{t-1}\cdots P_s$. Then we have the following lemma (\cite{rnv2010}).
\begin{lem}\label{network_nedic_lem}
	Let Assumption \ref{graph_ass} hold, then for all $i,j\in \huaV$ and all $t,s$ satisfying $t\geq s\geq 0$, it holds that $$\left|[Q(t,s)]_{ij}-\f{1}{m}\right|\leq \ooo\gamma^{t-s},$$
	in which $\ooo=(1-\f{\zeta}{4m^2})^{-2}$ and $\gamma=(1-\f{\zeta}{4m^2})^{\f{1}{B}}$.
\end{lem}
Now we are ready to state the following basic result.
\begin{lem}\label{consensus_lemma}
	Under Assumptions \ref{graph_ass} and \ref{gradient_ass}, it holds that
	\bes
	\left\|f_{i,t}-\bar f_{t}\right\|_\huaB\leq\ooo\gamma^{t-1}\sum_{i=1}^m\|f_{i,1}\|_\huaB+\f{m\ooo G}{\sigma_{\mb \Psi}}\sum_{s=1}^{t-1}\gamma^{t-s}\eta.
	\ees
\end{lem}
\begin{proof}
	We set $p_{i,t}=\wh f_{i,t+1}-f_{i,t}$,
	according to the algorithm \eqref{DFMD_algorithm}, we have
	$f_{i,t+1}=\sum_{j=1}^m[P_t]_{ij}\wh f_{i,t+1}=\sum_{j=1}^m[P_t]_{ij}(f_{i,t}+p_{i,t})$,
	direct iteration implies 
	$f_{i,t}=\sum_{j=1}^m[Q(t-1,1)]_{ij}f_{j,1}+\sum_{s=1}^{t-1}\sum_{j=1}^m[Q(t-1,\ell)]_{ij}p_{j,s}$,
	where $Q(t,\ell)=P_tP_{t-1}\cdots P_\ell$, $t\geq\ell\geq1$ is the transition matrix of $P_t$. Taking the average on both sides of the above equation, we have $\bar f_{t}=\bar f_1+\sum_{s=1}^{t-1}\f{1}{m}\sum_{i=1}^mp_{i,s}$.
	Subtraction between the above two equations and taking Banach norm on both sides, we have
	\bes
	\beal
	\left\|f_{i,t}-\bar f_{t}\right\|_\huaB\leq& \ooo\gamma^{t-1}\sum_{i=1}^m\|f_{i,1}\|_\huaB+\sum_{s=1}^{t-1}\ooo\gamma^{t-s}\sum_{j=1}^m\|p_{j,s}\|_\huaB\\
	\leq&\ooo\gamma^{t-1}\sum_{i=1}^m\|f_{i,1}\|_\huaB+\f{m\ooo G}{\sigma_{\mb \Psi}}\sum_{s=1}^{t-1}\gamma^{t-s}\eta.
	\eeal
	\ees
	The proof is finished.
\end{proof}

In the following analysis, we denote 
\bea
&&\hspace{-1cm}V_t=\max_{{s\in[t]}}\left\{\xi_T,v_{{s}}\right\},  \label{v_def1}\\
&&\hspace{-1cm}v_t=\max_{i\in\huaV}\left\{\sqrt{D_{\mb \Psi}(f^*\|\wh f_{i,{t}})}\right\}, \label{v_def2}
\eea
with $\xi_T>0$ a constant that may depend on the total iteration $T$ and will be determined in the subsequent analysis. Here, it is easy to see $v_1=\max_{i\in\huaV}\{\sqrt{D_{\mb \Psi}(f^*\|\wh f_{i,1})}\}$.

\begin{pro}\label{zong_ine}
	There holds that, for any $K\in[T]$,
	\bes
	\huaT_{1,K}\leq\huaT_{2,K}+\huaT_{3,K}
	\ees
	where $\huaT_{1,K}=\sum_{t=1}^K\f{\eta}{V_t}\sum_{i=1}^m\left\nn f_{i,t}-f^*,\partial_{f_{i,t}}\mb J_i\right\mm_{\huaB\times\huaB^*}$, $\huaT_{2,K}=\sum_{t=1}^K\f{1}{V_t}\sum_{i=1}^m\Big[D_{\mb \Psi}(f^*\|f_{i,t})-D_{\mb \Psi}(f^*\|\wh f_{i,t+1})\Big]$, $\huaT_{3,K}=\sum_{t=1}^K\sum_{i=1}^m\f{\eta^2}{2\sigma_{\mb \Psi}V_t}\left\|\partial_{f_{i,t}}\mb J_i\right\|_{\huaB^*}^2$. 
\end{pro}
\begin{proof}
Applying the first-order optimality again, we know $\nn  g-\wh f_{i,t+1}, \partial_{\wh f_{i,t+1}}D_{\mb \Psi}(\cdot,\partial_{h_{i,t}}\mb \Psi^*) \mm_{\huaB\times\huaB^*}\geq0,  \forall g\in\huaW$. Setting $g=f^*$, we have
\bes
\left\nn \wh f_{i,t+1}-f^*,\partial_{\wh f_{i,t+1}}\mb \Psi-\partial_{\partial_{h_{i,t}}\mb \Psi^*}\mb \Psi\right\mm_{\huaB\times\huaB^*}\leq0, \  \forall g\in\huaW.
\ees
Using the relation $\partial_{\partial_{h_{i,t}}\mb \Psi^*}\mb \Psi=\partial_{f_{i,t}}\mb \Psi-\eta\partial_{f_{i,t}}\mb J_i$ again, we have
\bes
&&\left\nn\wh f_{i,t+1}-f^*,\partial_{f_{i,t}}\mb \Psi-\partial_{\wh f_{i,t+1}}\mb \Psi\right\mm_{\huaB\times\huaB^*}\\
&&\geq\left\nn\wh f_{i,t+1}-f^*,\eta\partial_{f_{i,t}}\mb J_i\right\mm_{\huaB\times\huaB^*}.
\ees
Using the three-point inequality, we obtain that the left hand side of the above inequality satisfies,
\bes
\beal
LHS=&D_{\mb \Psi}(f^*\|f_{i,t})-D_{\mb \Psi}(f^*\|\wh f_{i,t+1})-D_{\mb \Psi}(\wh f_{i,t+1}\|f_{i,t})\\
\leq&D_{\mb \Psi}(f^*\|f_{i,t})-D_{\mb \Psi}(f^*\|\wh f_{i,t+1})-\f{\sigma_{\mb \Psi}}{2}\left\|\wh f_{i,t+1}-f_{i,t}\right\|_\huaB^2.
\eeal
\ees
The right hand side satisfies,
\bes
\beal
RHS=&\left\nn\wh f_{i,t+1}-f_{i,t},\eta\partial_{f_{i,t}}\mb J_i\right\mm_{\huaB\times\huaB^*}\\
&+\eta\left\nn f_{i,t}-f^*,\partial_{f_{i,t}}\mb J_i\right\mm_{\huaB\times\huaB^*}\\
\geq&-\sqrt{\sigma_{\mb \Psi}}\|\wh f_{i,t+1}-f_{i,t}\|_\huaB\cdot\f{\eta}{\sqrt{\sigma_{\mb \Psi}}}\left\|\partial_{f_{i,t}}\mb J_i\right\|_{\huaB^*}\\
&+\eta\left\nn f_{i,t}-f^*,\partial_{f_{i,t}}\mb J_i\right\mm_{\huaB\times\huaB^*}\\
\geq&-\f{\eta^2}{2\sigma_{\mb \Psi}}\left\|\partial_{f_{i,t}}\mb J_i\right\|_{\huaB^*}^2-\f{\sigma_{\mb \Psi}}{2}\|\wh f_{i,t+1}-f_{i,t}\|_\huaB^2\\
&+\eta\left\nn f_{i,t}-f^*,\partial_{f_{i,t}}\mb J_i\right\mm_{\huaB\times\huaB^*}.
\eeal
\ees
Combining the above estimates for LHS and RHS, we arrive at
\bes
\beal
\eta\left\nn f_{i,t}-f^*,\partial_{f_{i,t}}\mb J_i\right\mm_{\huaB\times\huaB^*}\leq& D_{\mb \Psi}(f^*\|f_{i,t})-D_{\mb \Psi}(f^*\|\wh f_{i,t+1})\\
&+\f{\eta^2}{2\sigma_{\mb \Psi}}\left\|\partial_{f_{i,t}}\mb J_i\right\|_{\huaB^*}^2.
\eeal
\ees
Hence, after dividing both sides by $V_t$ and taking summation from $i=1$ to $i=m$ and $t=1$ to $t=T$, we obtain the desired result.
\end{proof}

\begin{pro}\label{huaT_1_est}
Under Assumption \ref{graph_ass} and \ref{gradient_ass},	for any $K\in[T]$, there holds
\bes
 \huaT_{1,K}\geq\sum_{t=1}^K\f{\eta}{V_t}\Big[\mb J(f_{\ell,t})-\mb J(f^*)\Big]-\Theta_1\eta-\Theta_2T\eta^2, 
\ees
with $\Theta_1=\f{2\ooo L}{(1-\gamma)v_1}\sum_{i=1}^m\|f_{i,1}\|_\huaB$, $\Theta_2=\f{2\ooo m L^2}{\sigma_{\mb \Psi}(1-\gamma)v_1}$.
\end{pro}
\begin{proof}
Since $\mb J_i$ is a convex functional, there holds, for any $\ell\in\huaV$, $\left\nn f_{i,t}-f^*,\partial_{f_{i,t}}\mb J_i\right\mm_{\huaB\times\huaB^*}\geq\mb J_i(f_{i,t})-\mb J_i(f^*)=\mb J_i(f_{i,t})-\mb J_i(f_{\ell,t})+\mb J_i(f_{\ell,t})-\mb J_i(f^*)$.
After taking summation on both sides and using the fact $\mb J=\sum_{i=1}^m\mb J_i$, it holds that, for any $\ell\in\huaV$ and $K\in[T]$, $\huaT_{1,K}\geq\huaR_{1,K}+\huaR_{2,K}$, where
\bes
\beal
\huaR_{1,K}=&\sum_{t=1}^K\f{\eta}{V_t}\sum_{i=1}^m\Big[\mb J_i(f_{i,t})-\mb J_i(f_{\ell,t})\Big], \\
\huaR_{2,K}=&\sum_{t=1}^K\f{\eta}{V_t}\Big[\mb J(f_{\ell,t})-\mb J(f^*)\Big].
\eeal
\ees
Here, due to the fact that $\|\partial_{f_{\ell,t}}\mb J_i\|_{\huaB^*}\leq G$, $\huaR_{1,K}$ satisfies
\bes
\huaR_{1,K}\geq-G\sum_{t=1}^K\f{\eta}{V_t}\sum_{i=1}^m\left\|f_{i,t}-f_{\ell,t}\right\|_\huaB.
\ees
According to Lemma \ref{consensus_lemma} and Minkowski inequality for Banach space, we know, for any $K\in[T]$,
\bes
\beal
\huaR_{1,K}\geq&-G\sum_{t=1}^K\f{\eta}{v_1}\left(2\ooo\gamma^{t-1}\sum_{i=1}^m\|f_{i,1}\|_\huaB+\f{2m\ooo L}{\sigma_{\mb \Psi}}\sum_{s=1}^{t-1}\gamma^{t-s}\eta\right)\\
\geq&-\left[\f{2\ooo L}{(1-\gamma)v_1}\sum_{i=1}^m\|f_{i,1}\|_\huaB\right]\eta-\left[\f{2\ooo m L^2}{\sigma_{\mb \Psi}(1-\gamma)v_1}\right]T\eta^2.
\eeal
\ees
Combining this estimate with the estimate for $\huaR_{2,K}$, we obtain the desired result.
\end{proof}
The estimate for $\huaT_{2,K}$ is given in the following proposition.

\begin{pro} \label{huaT_2_est}
	Under Assumptions \ref{graph_ass} and \ref{Bregman_convexity_ass},	for any $K\in[T]$, it holds that
	\bes
	&&\huaT_{2,K}\leq\f{1}{v_1}\sum_{j=1}^mD_{\mb \Psi}(f^*\|\wh f_{j,1})-\f{1}{V_{K}}\sum_{j=1}^mD_{\mb \Psi}(f^*\|\wh f_{j,K+1}).
	\ees
\end{pro}
\begin{proof}
	According to the algorithm structure and double stochasticity of the communication matrix $P_t$, we have
	\bes
	&&\hspace{-0.5cm}\huaT_{2,K}=\sum_{t=1}^K\f{1}{V_t}\sum_{i=1}^m\Big[D_{\mb \Psi}(f^*\|\sum_{j=1}^m[P_t]_{ij}\wh f_{j,t})-D_{\mb \Psi}(f^*\|\wh f_{i,t+1})\Big]\\
	&&\hspace{-0.5cm}\leq\sum_{t=1}^K\f{1}{V_t}\sum_{i=1}^m\Big[\sum_{j=1}^m[P_t]_{ij}D_{\mb \Psi}(f^*\|\wh f_{j,t})-D_{\mb \Psi}(f^*\|\wh f_{i,t+1})\Big]\\
	&&\hspace{-0.5cm}=\sum_{t=1}^K\f{1}{V_t}\Big[\sum_{j=1}^mD_{\mb \Psi}(f^*\|\wh f_{j,t})-\sum_{i=1}^mD_{\mb \Psi}(f^*\|\wh f_{i,t})\Big].
	\ees
	The above inequality can be further bounded by
	\bes
	&&\f{1}{V_1}\sum_{j=1}^mD_{\mb \Psi}(f^*\|\wh f_{j,1})-\f{1}{V_K}\sum_{i=1}^mD_{\mb \Psi}(f^*\|\wh f_{i,K+1})\\
	&&+\sum_{t=2}^K\left(\f{1}{V_{t}}-\f{1}{V_{t-1}}\right)\sum_{i=1}^m D_{\mb \Psi}(f^*\|\wh f_{i,{t}}).
	\ees
	We know that, for  non-decreasing sequence $\{V_t\}$, it holds that $\f{1}{V_{t}}-\f{1}{V_{t-1}}\leq0$, $t\geq2$. Hence, after simplification, we have 
	\bes
	\huaT_{2,K}\leq\f{1}{V_1}\sum_{j=1}^mD_{\mb \Psi}(f^*\|\wh f_{j,1})-\f{1}{V_K}\sum_{i=1}^mD_{\mb \Psi}(f^*\|\wh f_{i,K+1}),
	\ees
	which completes the proof after noting that $v_1\leq V_1$.
\end{proof}

The estimate for $\huaT_{3,K}$ is given in the following proposition.
\begin{pro}\label{huaT_3_est}
Under Assumption \ref{gradient_ass},	for any $K\in[T]$ and any positive  $\xi_T>0$, there holds
	\bes
	\huaT_{3,K}\leq\xi_T\vee\left(\f{1}{\xi_T}\f{mG^2}{2\sigma_{\mb \Psi}}\eta^2T\right).
	\ees
\end{pro}
\begin{proof}
	According to the representation of $\huaT_{3,K}$, noting the definition of $V_t$, and $K\leq T$,  we have
	\bes
	\huaT_{3,K}\leq\sum_{t=1}^K\sum_{i=1}^m\f{\eta^2}{2\sigma_{\mb \Psi}V_t}G^2\leq\xi_T\vee\left(\f{1}{\xi_T}\f{mG^2}{2\sigma_{\mb \Psi}}\eta^2T\right),
	\ees
	which completes the proof.
	\end{proof}

For any $\ell\in\huaV$, define the following sequence $\ww f_{\ell,T}=\f{1}{T}\sum_{t=1}^Tf_{\ell,t}$.
The next theorem gives our first main result on the convergence performance of DFMD.
\begin{thm}
	Under Assumptions \ref{graph_ass}-\ref{Bregman_convexity_ass}, if the stepsize $\eta=\f{1}{\sqrt{T}}$, we have
	\bes
	\mb J(\ww f_{\ell,T})-\mb J(f^*)\leq\huaO\left(\f{1}{\sqrt{T}}\right).
	\ees
\end{thm}
\begin{proof}
Combining Proposition \ref{zong_ine} with Proposition \ref{huaT_1_est}-\ref{huaT_3_est}, we have, for any $\ell\in\huaV$ and any $K\in[T]$, 
\bes
&&\hspace{-0.5cm}\sum_{t=1}^K\f{\eta}{V_t}\Big[\mb J(f_{\ell,t})-\mb J(f^*)\Big]+\f{1}{V_K}\sum_{i=1}^mD_{\mb \Psi}(f^*\|\wh f_{i,K+1})\\
&&\hspace{-0.7cm}\leq \f{1}{v_1}\sum_{j=1}^mD_{\mb \Psi}(f^*\|\wh f_{j,1})+\xi_T\vee\left(\f{1}{\xi_T}\f{mG^2}{2\sigma_{\mb \Psi}}\eta^2T\right)+\Theta_1\eta+\Theta_2T\eta^2.\\
&&\hspace{-0.7cm}=:S_T.
\ees
From this inequality and the definition of $v_{K+1}$ in \eqref{v_def2}, we know that for any $K\in[T]$, it holds that
$\f{v_{K+1}^2}{V_K}\leq \f{1}{V_K}\sum_{i=1}^mD_{\mb \Psi}(f^*\|\wh f_{i,K+1})\leq S_T$,
which indicates that $v_{K+1}^2\leq V_KS_T$. Note that when $K=1$, due to 
$v_1\leq\f{1}{v_1}\sum_{j=1}^mD_{\mb \Psi}(f^*\|\wh f_{j,1})\leq S_T$,
 we know $V_1=\max\{v_1,\xi_T\}\leq S_T$. Now suppose that for some $K\in[T]$, $V_K\leq S_T$, then we have $V_{K+1}=\max\{V_K,v_{K+1}\}\leq\max\{V_K,\sqrt{V_KS_T}\}\leq S_T$. Hence an induction indicates that, for any $K\in[T]$, $V_K\leq S_T$. Then for any $\ell\in\huaV$ and any $K\in[T]$, 
 \bes
 &&\hspace{-0.5cm}\sum_{t=1}^K\Big[\mb J(f_{\ell,t})-\mb J(f^*)\Big]\leq\f{V_KS_T}{\eta}\leq\f{S_T^2}{\eta}.
 \ees
Specifically, when $K=T$, using the convexity of the functional $\mb J$ and dividing both sides of the above inequality by $T$, we have, for any $\ell\in\huaV$, $\mb J(\ww f_{\ell, T})-\mb J(f^*)\leq\f{S_T^2}{T\eta}$.
Taking $\xi_T=\sqrt{\f{mG^2\eta^2T}{2\sigma_{\mb \Psi}}}$, we have
\bes
&&\hspace{-0.8cm}\mb J(\ww f_{\ell, T})-\mb J(f^*)\\
&&\hspace{-0.8cm}\leq\f{1}{T\eta}\Bigg[ \f{1}{V_1}\sum_{j=1}^mD_{\mb \Psi}(f^*\|\wh f_{j,1})+\sqrt{\f{mG^2}{2\sigma_{\mb \Psi}}}\eta\sqrt{T}+\Theta_1\eta+\Theta_2T\eta^2\Bigg]^2.
\ees
After taking  stepsize $\eta=\f{1}{\sqrt{T}}$, we finally have, for any $\ell\in\huaV$,
\bes
\mb J(\ww f_{\ell,T})-\mb J(f^*)\leq\huaO\left(\f{1}{\sqrt{T}}\right).
\ees
We finish the proof.
\end{proof}

Till now, we have provided a complete convergence theory for DFMD. We have proven that the local ergodic sequence can achieve an optimal convergence rate of $\huaO\left(\f{1}{\sqrt{T}}\right)$ within the general framework of Banach spaces. It is worth noting that this analytical process does not require any boundedness conditions induced by the Banach norm for the decision space, even in the most general Banach spaces. Such conditions are typically necessary in most existing works within Euclidean spaces. In the next section, we will concentrate on the nonconvex optimization problems on Hilbert spaces and study the convergence theory of an effective DFGD method.

\section{Distributed functional nonconvex optimization over Hilbert spaces}
In this section, we study nonconvex DFO problem-related to the model \eqref{problem_formulation} when the local objective functions $\mb J_i$, $i\in\huaV$ can be nonconvex. We concentrate on the setting that the underlying space is a  Hilbert space
$\huaB=\huaH$.  In this section, without loss of generality and considering some typical important cases, we consider the case that the local objective functionals $\mb J_i$, $i\in\huaV$  are Fr\'echet differentiable as in Definition \ref{frechet_def}.  we also assume the following condition on the $L$-smoothness of $\mb J_i$, $i\in\huaV$ in order to delve  deeply into the nonconvex DFO problem.
\begin{ass} \label{Lsmooth_ass}
The	local objective functionals $\mb J_i$, $i\in\huaV$  are Fr\'echet differentiable on $\huaH$ and $L$-smooth.
\end{ass}
The $L$-smoothness condition has been widely adopted in a lot of existing work related to optimization and learning. 

For solving the nonconvex DFO problem \eqref{problem_formulation}, we aim to study the convergence theory of the  distributed functional gradient descent algorithm (DFGD) given by 
\bea
\textbf{DFGD}\left\{
\begin{aligned}
	h_{i,t}=&f_{i,t}-\eta\mb D_{f_{i,t}}\tb J_i,\\
	f_{i,t+1}=&\sum_{j=1}^m[P_t]_{ij}	h_{j,t}.
\end{aligned}
\right. 
\eea
Since for a real Hilbert space $\huaH$, Riesz representation theorem indicates that, for any continuous linear functional $\mb F\in\huaH$, there is a unique $f\in\huaH$ such that $\mb F(g)=\left\nn g, f\right\mm_\huaH$, which indicates $\huaH\cong\huaH^*$. Namely, the dual space of $\huaH$ is itself in the sense of isometric isomorphism. Hence, we know the Fr\'echet derivative $\mb D_{f_{i,t}}\mb J$ of $\mb J$ naturally serves an element of $\huaH$ which shows that the DFGD structure given above makes sense.

To establish the convergence theory of DFGD, we start with the following result on the estimate of $\mb J(\bar f_{t+1})-\mb J(\bar f_t)$.
\begin{pro}\label{basic_pro}
	Under Assumptions \ref{graph_ass} and \ref{Lsmooth_ass}, it holds that
\bes
&&\mb J(\bar f_{t+1})-\mb J(\bar f_t)\\
&&\leq\left(\f{\eta}{2}+L\eta^2\right)\sum_{i=1}^m\left\|\mb D_{f_{i,t}}\mb J_i-\mb D_{\bar f_t}\mb J_i\right\|_\huaH^2\\
&&-\left(\f{\eta}{2m}-\f{L\eta^2}{m}\right)\|\mb D_{\bar f_t}\mb J\|_\huaH^2.
\ees
\end{pro}
\begin{proof}
Since the functionals $\mb J_i$, $i\in\huaV$ are $L$-smooth, then it is easy to verify, $\mb J$ is $mL$-smooth. Accordingly, there holds
\bes
\mb J(\bar f_{t+1})\leq \mb J(\bar f_t)-\left\nn \bar f_{t+1}-\bar f_t, \mb D_{\bar f_t}\mb J\right\mm_\huaH+\f{mL}{2}\left\|\bar f_{t+1}-\bar f_t\right\|_\huaH^2.
\ees
According to the double stochasticity of the communication matrix $P_t$, we know
$\sum_{i=1}^mf_{i,t+1}=\sum_{i=1}^m\sum_{j=1}^m[P_t]_{ij}	h_{j,t}=\sum_{i=1}^mh_{i,t}$,
which further indicates
$\bar f_{t+1}-\bar f_t=-\eta\f{1}{m}\sum_{i=1}^m\mb D_{f_{i,t}}\tb J_i$. Then it follows that
\bes
&&\left\nn \bar f_{t+1}-\bar f_t, \mb D_{\bar f_t}\mb J\right\mm_\huaH\\
&&=-\eta\left\nn\f{1}{m}\sum_{i=1}^m\mb D_{f_{i,t}}\mb J_i,\mb D_{\bar f_t}\mb J\right\mm_\huaH\\
&&=-\f{\eta}{m}\sum_{i=1}^m\left\nn\mb D_{f_{i,t}}\mb J_i-\mb D_{\bar f_t}\mb J_i+\mb D_{\bar f_t}\mb J_i,\mb D_{\bar f_t}\mb J\right\mm_\huaH\\
&&=-\f{\eta}{m}\sum_{i=1}^m\left\nn \mb D_{f_{i,t}}\mb J_i-\mb D_{\bar f_t}\mb J_i,\mb D_{\bar f_t}\mb J\right\mm_\huaH-\f{\eta}{m}\left\|\mb D_{\bar f_t}\mb J\right\|_\huaH^2\\
&&\leq\f{\eta}{m}\sum_{i=1}^m\|\mb D_{f_{i,t}}\mb J_i-\mb D_{\bar f_t}\mb J_i\|_\huaH\cdot\|\mb D_{\bar f_t}\mb J\|_\huaH-\f{\eta}{m}\left\|\mb D_{\bar f_t}\mb J\right\|_\huaH^2.
\ees
Applying Young's inequality to the first term of the above inequality, we have
\bes
&&\f{\eta}{m}\sum_{i=1}^m\|\mb D_{f_{i,t}}\mb J_i-\mb D_{\bar f_t}\mb J_i\|_\huaH\cdot\|\mb D_{\bar f_t}\mb J\|_\huaH\\
&&\leq\f{\eta}{2}\sum_{i=1}^m\|\mb D_{f_{i,t}}\mb J_i-\mb D_{\bar f_t}\mb J_i\|_\huaH^2+\f{\eta}{2m}\|\mb D_{\bar f_t}\mb J\|_\huaH^2.
\ees
Combining this inequality with the above inequality, we have
\bes
&&\Big\nn \bar f_{t+1}-\bar f_t, \mb D_{\bar f_t}\mb J\Big\mm_\huaH\\
&&\leq\f{\eta}{2}\sum_{i=1}^m\|\mb D_{f_{i,t}}\mb J_i-\mb D_{\bar f_t}\mb J_i\|_\huaH^2-\f{\eta}{2m}\|\mb D_{\bar f_t}\mb J\|_\huaH^2.
\ees
On the other hand, the following estimates hold
\bes
&&\f{mL}{2}\left\|\bar f_{t+1}-\bar f_t\right\|_\huaH^2\\
&&\leq\f{mL\eta^2}{2}\left\|\f{1}{m}\sum_{i=1}^m\mb D_{f_{i,t}}\tb J_i\right\|_\huaH^2\\
&&=\f{L\eta^2}{2m}\left\|\sum_{i=1}^m\Big[\mb D_{f_{i,t}}\mb J_i-\mb D_{\bar f_t}\mb J_i\Big]+\mb D_{\bar f_t}\mb J\right\|_\huaH^2\\
&&\leq\f{L\eta^2}{m}\left\|\sum_{i=1}^m\Big[\mb D_{f_{i,t}}\mb J_i-\mb D_{\bar f_t}\mb J_i\Big]\right\|_\huaH^2+\f{L\eta^2}{m}\left\|\mb D_{\bar f_t}\mb J\right\|_\huaH^2\\
&&\leq L\eta^2\sum_{i=1}^m\left\|\mb D_{f_{i,t}}\mb J_i-\mb D_{\bar f_t}\mb J_i\right\|_\huaH^2+\f{L\eta^2}{m}\left\|\mb D_{\bar f_t}\mb J\right\|_\huaH^2.
\ees
After combining these estimates, we obtain the desired result.
\end{proof}

In the following, for technical consideration, we require the well-known Polyak-\L ojasiewicz (P\L) condition for Fr\'echet differentiable functional.
\begin{defi}
	A real functional $\mb F:\huaH\rightarrow\mbb R$ is called $\mu$-Polyak-\L ojasiewicz ($\mu$-P\L) if for some constant $\mu>0$,
	\bes
	\f{1}{2}\left\|\mb D_f \mb F\right\|_\huaH^2\geq\mu\left(\mb F(f)-\mb F(f^*)\right), \ \forall f\in \huaH,
	\ees
	where $f^*$ is the global minimizer of the functional $\mb F$.
\end{defi}

Inspired by the convergence expressions in classical non-convex optimization in Euclidean space, for any agent $\ell\in\huaV$, we use the average rate $\f{1}{T}\sum_{t=1}^T\|\mb D_{f_{\ell,t}}\mb J\|_{\huaH}^2$ to describe the convergence performance of DFGD in the framework of Hilbert spaces in the next theorem.
\begin{thm}
Under Assumptions \ref{graph_ass}, \ref{gradient_ass}, and \ref{Lsmooth_ass}. If the stepsize $\eta$ satisties	$\eta=\f{1}{2L\sqrt{T+3}}$, then for any $\ell\in\huaV$, we have
	\bes
	\f{1}{T}\sum_{t=1}^T\|\mb D_{f_{\ell,t}}\mb J\|_{\huaH}^2\leq\huaO\left(\f{1}{\sqrt{T}}\right).
	\ees
	Moreover, when the global functional $\mb J$ satisfies  $\mu$-Polyak-\L ojasiewicz condition, then for any $\ell\in\huaV$,
	\bes
	\f{1}{T}\sum_{t=1}^T\Big[\mb J(f_{\ell,t})-\mb J(f^*)\Big]\leq\huaO\left(\f{1}{\sqrt{T}}\right).
	\ees
\end{thm}

\begin{proof}
	According to the selection of the stepsize $\eta$, we know that $\eta\leq\f{1}{4L}$. Hence based on Proposition \ref{basic_pro}, a simple computation yields that
\bes
\beal
\f{1}{4m}\eta\|\mb D_{\bar f_t}\mb J\|_\huaH^2\leq&\Big[\mb J(\bar f_t)-\mb J(\bar f_{t+1})\Big]\\
&+\f{3}{4}\eta\sum_{i=1}^m\left\|\mb D_{f_{i,t}}\mb J_i-\mb D_{\bar f_t}\mb J_i\right\|_\huaH^2.
\eeal
\ees
After taking summation from $t=1$ to $t=T$ and using the $L$-smoothness of the functional $\mb J_i$ and the fact $\mb J(f^*)\leq\mb J(\bar f_{T+1})$, we have
\bes
\beal
\f{1}{4m}\eta\sum_{t=1}^T\|\mb D_{\bar f_t}\mb J\|_\huaH^2\leq&\mb J(\bar f_1)-\mb J(f^*)\\
&+\f{3}{4}L^2\eta\sum_{t=1}^T\sum_{i=1}^m\left\|f_{i,t}-\bar f_t\right\|_\huaH^2.
\eeal
\ees
Applying Lemma \ref{consensus_lemma} to the case that $\huaB=\huaW=\huaH$ (recall that, in the standard Hilbert norm topology, it is both open and closed as a whole) and $\mb \Psi(f)=\f{1}{2}\|f\|_\huaH^2$, we know
\bes
\beal
\left\|f_{i,t}-\bar f_{t}\right\|_\huaH
\leq&\ooo\gamma^{t-1}\sum_{i=1}^m\|f_{i,1}\|_\huaH+m\ooo G\sum_{s=1}^{t-1}\gamma^{t-s}\eta\\
\leq&\ooo\gamma^{t-1}\sum_{i=1}^m\|f_{i,1}\|_\huaH+\f{m\ooo G}{1-\gamma}\eta.
\eeal
\ees
Then it follows that
\bea
\left\|f_{i,t}-\bar f_{t}\right\|_\huaH
^2\leq2\ooo^2\Big[\sum_{i=1}^m\|f_{i,1}\|_\huaH\Big]^2\gamma^{2(t-1)}+2\Big[\f{m\ooo G}{1-\gamma}\Big]^2\eta^2.\label{consensus_square_basic}
\eea
This estimate further indicates that
\bea
\beal
\sum_{t=1}^T\sum_{i=1}^m\left\|f_{i,t}-\bar f_{t}\right\|_\huaH^2\leq&\f{2m\ooo^2\Big[\sum_{i=1}^m\|f_{i,1}\|_\huaH\Big]^2}{1-\gamma^2}\\
&+2m\Big[\f{m\ooo G}{1-\gamma}\Big]^2\eta^2T. \label{sum_consensus}
\eeal
\eea
Combining above estimates, we arrive at
\bes
&&\eta\sum_{t=1}^T\|\mb D_{\bar f_t}\mb J\|_\huaH^2\leq4m\Big[\mb J(\bar f_1)-\mb J(f^*)\Big]\\
&&+\f{6m^2\ooo^2\Big[\sum_{i=1}^m\|f_{i,1}\|_\huaH\Big]^2L^2\eta}{1-\gamma^2}+6m^2\Big[\f{m\ooo G}{1-\gamma}\Big]^2L^2\eta^3T.\\
\ees
After setting the constants $\Delta_1=4m[\mb J(\bar f_1)-\mb J(f^*)]$, $\Delta_2=\f{6m^2\ooo^2[\sum_{i=1}^m\|f_{i,1}\|_\huaH]^2L^2}{1-\gamma^2}$, $\Delta_3=\f{3}{2}m^2[\f{m\ooo G}{1-\gamma}]^2L$,
we have
\bes
\sum_{t=1}^T\|\mb D_{\bar f_t}\mb J\|_\huaH^2\leq\f{\Delta_1}{\eta}+\Delta_2+\Delta_3\eta T.
\ees
On the other hand, for any $\ell\in\huaV$, it holds that
\bes
&&\sum_{t=1}^T\|\mb D_{f_{\ell,t}}\mb J\|_{\huaH}^2=\sum_{t=1}^T\left\|\Big[\mb D_{f_{\ell,t}}\mb J-\mb D_{\bar f_{t}}\mb J\Big]+\mb D_{\bar f_{t}}\mb J\right\|_{\huaH}^2\\
&&\leq\sum_{t=1}^T2\|\mb D_{f_{\ell,t}}\mb J-\mb D_{\bar f_{t}}\mb J\|_{\huaH}^2+2\sum_{t=1}^T\|\mb D_{\bar f_{t}}\mb J\|_{\huaH}^2\\
&&\leq\sum_{t=1}^T2L^2\|f_{\ell,t}-\bar f_{t}\|_{\huaH}^2+2\sum_{t=1}^T\|\mb D_{\bar f_{t}}\mb J\|_{\huaH}^2.
\ees
In the last inequality, we have used the basic equivalence mentioned after Definition \ref{Lsmooth_def}. By following the same procedures of getting \eqref{sum_consensus}, we have
$\sum_{t=1}^T\left\|f_{i,t}-\bar f_{t}\right\|_\huaH^2\leq\Delta_4+\Delta_5\eta^2T$
with $\Delta_4=\f{2\ooo^2[\sum_{i=1}^m\|f_{i,1}\|_\huaB]^2}{1-\gamma^2}$, $\Delta_5=2[\f{m\ooo G}{1-\gamma}]^2$. 
Combining the above estimates, we finally arrive at, for any $\ell\in\huaV$,
\bes
\sum_{t=1}^T\|\mb D_{f_{\ell,t}}\mb J\|_{\huaH}^2\leq\f{2\Delta_1}{\eta}+2\Delta_2+2\Delta_3\eta T+2L^2\Delta_4+2L^2\Delta_5\eta^2T.
\ees
Divide both sides by $T$, we have
\bes
\f{1}{T}\sum_{t=1}^T\|\mb D_{f_{\ell,t}}\mb J\|_{\huaH}^2\leq\f{2\Delta_1}{T\eta}+\f{2\Delta_2}{T}+2\Delta_3\eta +\f{2L^2\Delta_4}{T}+2L^2\Delta_5\eta^2.
\ees
Due to the selection of the stepsize $\eta=\f{1}{2L\sqrt{T+3}}$, we have, for any $\ell\in\huaV$,
$\f{1}{T}\sum_{t=1}^T\|\mb D_{f_{\ell,t}}\mb J\|_{\huaH}^2\leq\huaO\left(\f{1}{\sqrt{T}}\right)$.
A direct result under  Polyak-\L ojasiewicz condition finally shows, for any $\ell\in\huaV$,
\bes
\f{1}{T}\sum_{t=1}^T\Big[\mb J(f_{\ell,t})-\mb J(f^*)\Big]\leq\huaO\left(\f{1}{\sqrt{T}}\right).
\ees
Thus, we have completed the proof.
\end{proof}

In the subsequent analysis, for  of the last iteration of the local sequence generated from DFGD associated with any agent, we study its related linear convergence rate.  We establish an $R$-linear convergence rate  up to a solution level that is proportional to stepsize.

\begin{thm}
Under Assumptions \ref{graph_ass}, \ref{gradient_ass}, and \ref{Lsmooth_ass}. If the global functional $\mb J$ satisfies the $\mu$-Polyak-\L ojasiewicz condition and the stepsize $\eta$ satisfies	$0<\eta<\min\{\f{2m}{\mu},\f{1}{4L}\}$, then for any $\ell\in\huaV$, we have
	$\mb J( f_{\ell,t})-\mb J(f^*)\leq C\nu^{t-1}+\Omega \eta$
	with $C,\Omega\in(0,\infty)$ and $\nu\in(0,1)$.
\end{thm}
\begin{proof}
	Proposition \ref{basic_pro} and the $\mu$-Polyak-\L ojasiewicz condition of $\mb J$ imply that
	\bes
	&&\Big[\mb J(\bar f_{t+1})-\mb J(f^*)\Big]-\Big[\mb J(\bar f_t)-\mb J(f^*)\Big]\\
	&&\leq\left(\f{\eta}{2}+L\eta^2\right)L^2\sum_{i=1}^m\left\|f_{i,t}-\bar f_t\right\|_\huaH^2\\
	&&\quad -\left(\f{\eta}{2m}-\f{L\eta^2}{m}\right)2\mu\Big[\mb J(\bar f_t)-\mb J(f^*)\Big].
	\ees
	Then, according to the range of the stepsize $\eta$, we have
	\bes
	&&\Big[\mb J(\bar f_{t+1})-\mb J(f^*)\Big]\\
	&&\leq\left[ 1-\left(\f{1}{2m}-\f{L\eta}{m}\right)2\mu\eta\right]\Big[\mb J(\bar f_t)-\mb J(f^*)\Big]\\
	&&\quad+\left(\f{\eta}{2}+L\eta^2\right)L^2\sum_{i=1}^m\left\|f_{i,t}-\bar f_t\right\|_\huaH^2\\
	&&\leq\left( 1-\f{\mu}{2m}\eta\right)\Big[\mb J(\bar f_t)-\mb J(f^*)\Big]+\f{3L^2}{4}\eta\sum_{i=1}^m\left\|f_{i,t}-\bar f_t\right\|_\huaH^2.
	\ees
	By directly using \eqref{consensus_square_basic}, we have
	\bes
	\f{3L^2}{4}\eta\sum_{i=1}^m\left\|f_{i,t}-\bar f_{t}\right\|_\huaH
	^2\leq\Delta_6\gamma^{2(t-1)}\eta+\Delta_7\eta^3,
	\ees
	with
$\Delta_6=	\f{6L^2\ooo^2m}{4}\Big[\sum_{i=1}^m\|f_{i,1}\|_\huaH\Big]^2$, $\Delta_7=\f{3L^2m}{2}\Big[\f{m\ooo G}{1-\gamma}\Big]^2$.
Substitute this estimate into the above inequality, it can be easily derived that
\bes
&&\Big[\mb J(\bar f_{t+1})-\mb J(f^*)\Big]\\
&&\leq\left( 1-\f{\mu}{2m}\eta\right)\Big[\mb J(\bar f_t)-\mb J(f^*)\Big]+\Delta_6\gamma^{2(t-1)}\eta+\Delta_7\eta^3.
\ees	
Direct iteration and using the fact that $1-\f{\mu}{2m}\eta<1$, we have
\bes
&&\hspace{-0.5cm}\Big[\mb J(\bar f_{t+1})-\mb J(f^*)\Big]\\
&&\hspace{-0.5cm}\leq\left( 1-\f{\mu}{2m}\eta\right)^2\Big[\mb J(\bar f_{t-1})-\mb J(f^*)\Big]+\left( 1-\f{\mu}{2m}\eta\right)\Delta_6\gamma^{2(t-2)}\eta\\
&&\hspace{-0.5cm}\quad+\left( 1-\f{\mu}{2m}\eta\right)\Delta_7\eta^3 +\Delta_6\gamma^{2(t-1)}\eta+\Delta_7\eta^3\\
&&\hspace{-0.5cm}\leq\left( 1-\f{\mu}{2m}\eta\right)^3\Big[\mb J(\bar f_{t-2})-\mb J(f^*)\Big]+\left( 1-\f{\mu}{2m}\eta\right)^2\Delta_6\gamma^{2(t-3)}\eta\\
&&\hspace{-0.5cm}\quad+\left( 1-\f{\mu}{2m}\eta\right)^2\Delta_7\eta^3+\left( 1-\f{\mu}{2m}\eta\right)\Delta_6\gamma^{2(t-2)}\eta\\
&&\hspace{-0.5cm}\quad+\left( 1-\f{\mu}{2m}\eta\right)\Delta_7\eta^3 +\Delta_6\gamma^{2(t-1)}\eta+\Delta_7\eta^3\\
&&\hspace{-0.5cm}\cdots\\
&&\hspace{-0.5cm}\leq\left( 1-\f{\mu}{2m}\eta\right)^t\Big[\mb J(\bar f_{1})-\mb J(f^*)\Big]+\f{\Delta_6}{1-\gamma^2}\eta+\f{2m\Delta_7}{\mu}\eta^2.
\ees	
That is to say,
\bes
&&\hspace{-0.5cm}\Big[\mb J(\bar f_{t})-\mb J(f^*)\Big]\\
&&\hspace{-0.5cm}	\leq\left( 1-\f{\mu}{2m}\eta\right)^{t-1}\Big[\mb J(\bar f_{1})-\mb J(f^*)\Big]+\f{\Delta_6}{1-\gamma^2}\eta+\f{2m\Delta_7}{\mu}\eta^2.
\ees
Then for any $\ell\in\huaV$, we have
\bes
&&\mb J( f_{\ell,t})-\mb J(f^*)\\
&&=\Big[\mb J( \bar f_t)-\mb J( f^*)\Big]+\Big[\mb J( f_{\ell,t})-\mb J( \bar f_t)\Big]\\
&&\leq \left( 1-\f{\mu}{2m}\eta\right)^{t-1}\Big[\mb J(\bar f_{1})-\mb J(f^*)\Big]\\
&&\quad+\f{\Delta_6}{1-\gamma^2}\eta+\f{2m\Delta_7}{\mu}\eta^2+G\left\|f_{\ell,t}-\bar f_t\right\|_\huaH\\
&&\leq\left( 1-\f{\mu}{2m}\eta\right)^{t-1}\Big[\mb J(\bar f_{1})-\mb J(f^*)\Big]\\
&&\quad+\f{\Delta_6}{1-\gamma^2}\eta+\f{2m\Delta_7}{\mu}\eta^2+\Delta_{8}\eta+\Delta_9\gamma^{t-1}
\ees
with 
$\Delta_8=\f{m\ooo G^2}{1-\gamma}$,  $\Delta_9=\ooo G\sum_{i=1}^m\|f_{i,1}\|_\huaH$.
After setting the constant
$\Omega=\f{\Delta_6}{1-\gamma^2}+\f{4m^2\Delta_7}{\mu^2}+\Delta_8$,
 we have
\bes
&&\mb J( f_{\ell,t})-\mb J(f^*)\\
&&\leq\left( 1-\f{\mu}{2m}\eta\right)^{t-1}\Big[\mb J(\bar f_{1})-\mb J(f^*)\Big]+\Omega\eta+\Delta_9\gamma^{t-1}.
\ees
After further taking 
$C=\max\left\{\mb J(\bar f_1)-\mb J(f^*),\Delta_9\right\}$, $\nu=\max\left\{ 1-\f{\mu}{2m}\eta,\gamma\right\}$,
we finally arrive at, for any $\ell\in\huaV$,
\bes
\mb J( f_{\ell,t})-\mb J(f^*)\leq C\nu^{t-1}+\Omega \eta
\ees
with $C\in(0,\infty)$ and $\nu\in(0,1)$. We finish the proof.
	\end{proof}

\section{Typical scenarios}\label{section5}
In this section, we introduce some  typical settings from  machine learning and statistical learning that are closely related to the theory established in this work. We provide the formulations related to the DFMD or DFGD algorithms over time-varying networks established in this work.  Although these algorithms have not been carefully studied independently in existing work, it has been ensured based on the convergence theory in this work, that these algorithms demonstrate good convergence performances and application values in these situations as direct corollaries.
\subsection{Distributed functional optimization over measure spaces}
In scenarios of machine learning, many problems can be formulated as optimizing a convex functional over a vector space of measures. For example, in Bayesian inference, it is typical to optimize the Kullback-Leibler divergence with respect to the target, which relates to the posterior distribution of the parameters of interest \cite{akl2022}. In distribution regression,  we often need to deal with regression schemes based on probability measures as samples \cite{yh2024}. In learning theory based on infinite-width one hidden layer neural network, the problem can often be reduced to the variational problem on the probability distributions over the neural network parameters \cite{cb2018}. These problems can be easily transformed into functional optimization models and solved using the time-varying multi-agent DFO algorithmic framework provided in this paper. 

A common framework can be described as follows: suppose that $X$ is a compact metric space of $\mbb R^n$ with the naturally induced Euclidean topology, consider the space of Radon measures ($\huaM_r(X)$, $\|\cdot\|_{\text{TV}}$) supported on $\huaC$ with the total variation (TV) norm $\|\cdot\|_{\text{TV}}$. The goal is to minimize a real convex functional $\mb J:\huaC\rightarrow\mbb R$, where $\huaC$ is a closed convex set of the space $\huaM_r(X)$. Here, we consider the setting that $\mb J$ can be decomposed as a summation of local functionals $\mb J_i$, $i\in\huaV$, which can be handled by a multi-agent system described in this paper. In practice, many problems are interested in the case that $\huaC=\huaP(X)$, the set of all  probability distributions on $X$ \cite{yhy2024}. If $\mb J_i$, $i\in\huaV$ are proper, strictly convex, and G\^ateaux differentiable functionals on $\huaM_r(X)$, $\partial_{\mu_{i,t}}\mb J_i$ serves as a bounded linear functional on $\huaM_r(X)$. According to our established theory, for an appropriately selected  mirror map $\mb \Psi$, our DFMD on measure space reduces to the form of\\

\noi  \textbf{MS-DFMD:}
\bes
\left\{
\begin{aligned}
	\nu_{i,t+1}=&\arg\min_{\mu\in\huaP(X)}\left\{\partial_{\mu_{i,t}}\mb J_i(\mu-\mu_{i,t})+\f{1}{\eta}D_{\mb \Psi}(\mu\|\mu_{i,t})\right\},\\
	\mu_{i,t+1}=&\sum_{j=1}^m[P_t]_{ij} \nu_{j,t+1}.
\end{aligned}  
\right. 
\ees
With the theory of Section \ref{DFMD} at hand, we know that the ergodic sequence $\ww \mu_{\ell,T}=\f{1}{T}\sum_{t=1}^T\mu_{\ell,t}$ of probability distributions generated from the above MS-DFMD algorithm is able to achieve a convergence rate of $\huaO(1/\sqrt{T})$ for seeking a target probability distribution $\mu^*$. This indicates that the algorithm MS-DFMD has potentially important application values in core areas of machine learning based on probability measure space as underlying space, such as distribution regression with  samples consisting of probability distributions \cite{yh2024} and functional neural networks over Wasserstein space \cite{syz2025}.

\begin{figure}[t] 
	\centering
	\begin{minipage}{0.2\textwidth} 
		\centering
		\includegraphics[width=\linewidth]{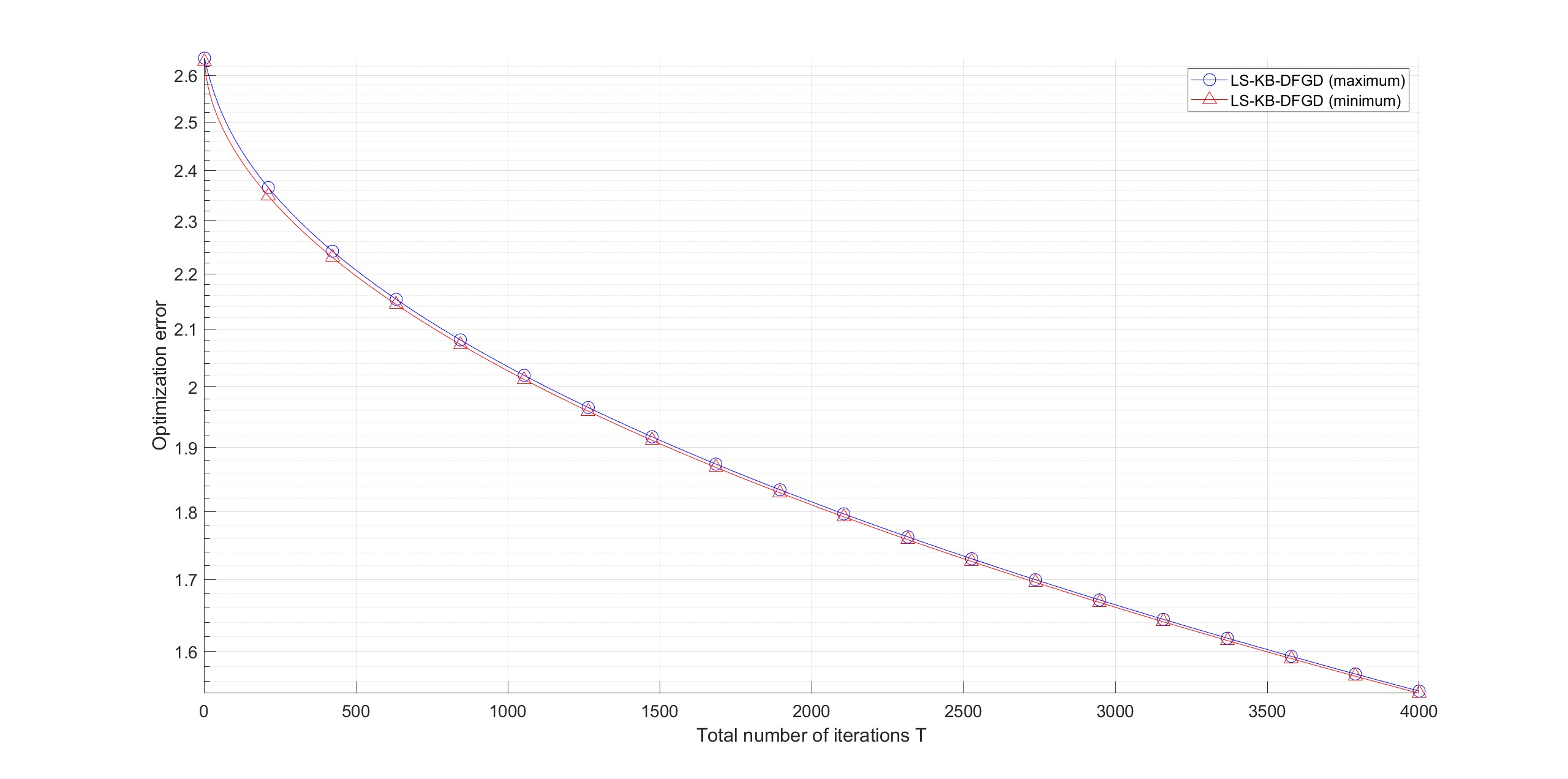}
		\caption{\tiny{Maximum and minimum of the optimization errors across all agents versus the total number
				of iterations $T$ of the LS-KB-DFGD algorithm.}}
		\label{figure1}
	\end{minipage}
	\hfill 
	\begin{minipage}{0.2\textwidth}
		\centering
		\includegraphics[width=\linewidth]{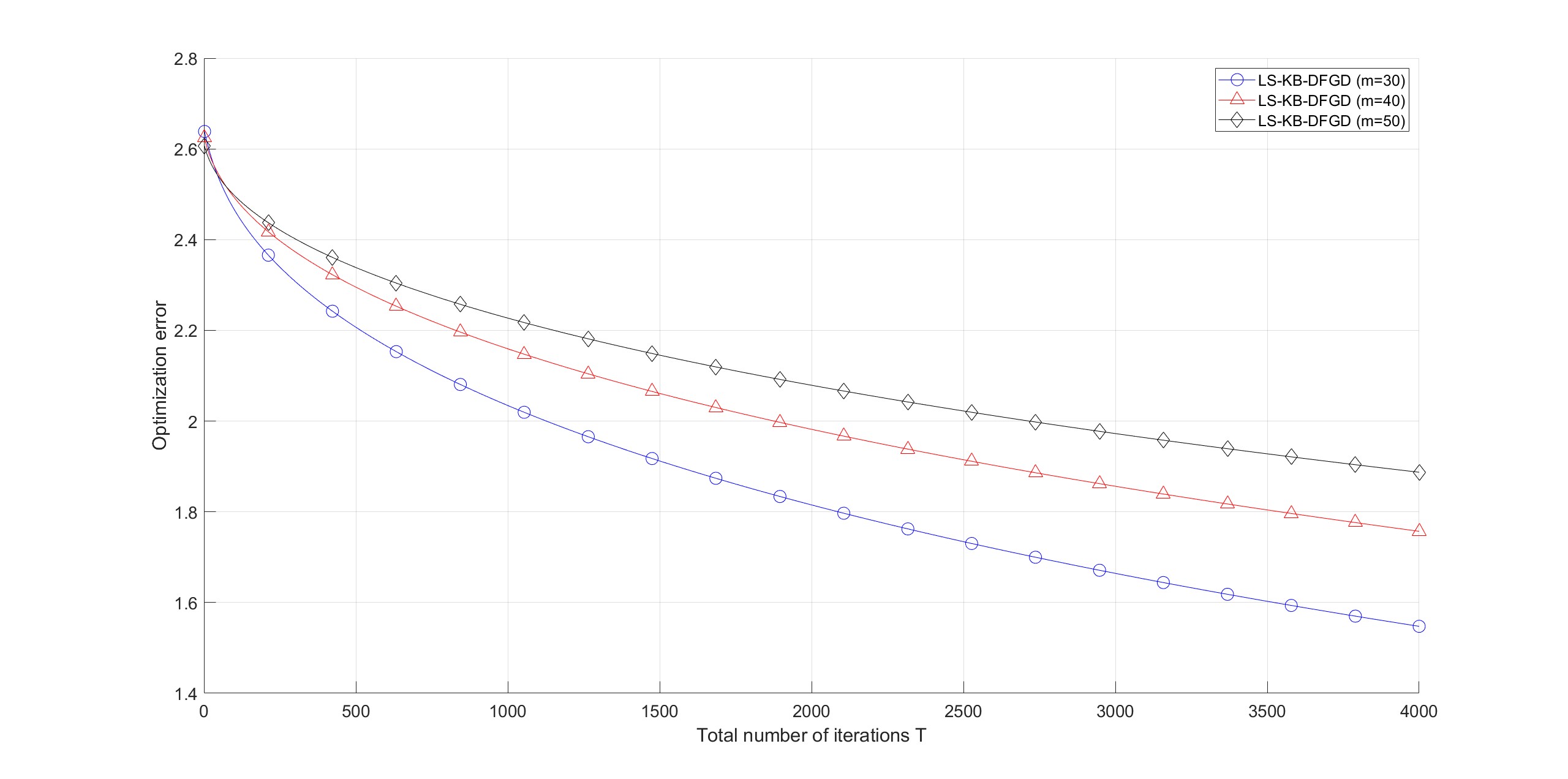}
		\caption{\tiny{Maximum of the optimization errors versus the total number of iterations $T$ of the LS-KB-DFGD algorithm for three different choices of the number of agents $m$.}}
		\label{figure2}
	\end{minipage}
	\hfill
	\begin{minipage}{0.2\textwidth}
		\centering
		\includegraphics[width=\linewidth]{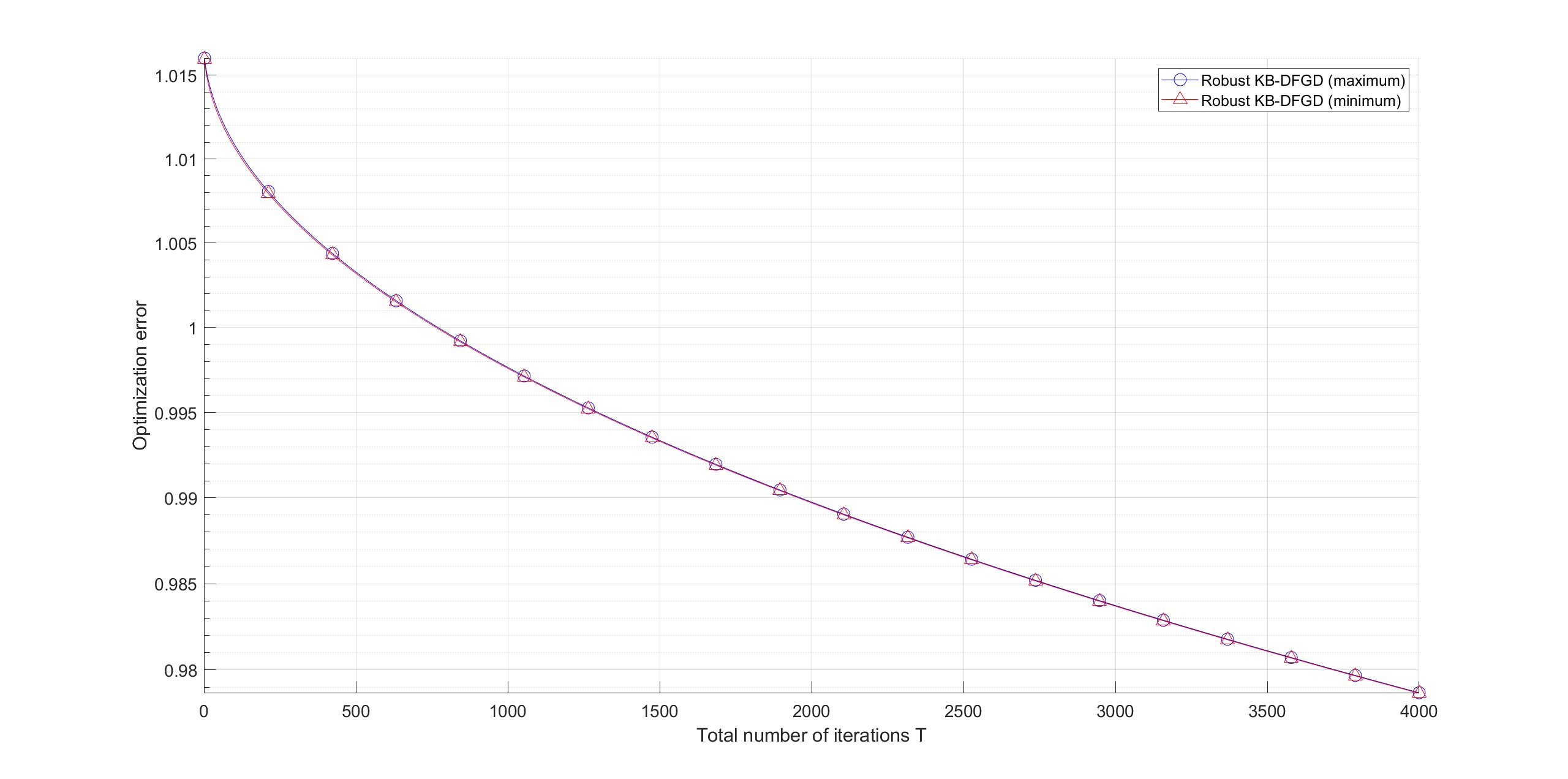}
		\caption{\tiny{Maximum and minimum of the optimization errors across all agents versus the total number
				of iterations $T$ of the Robust KB-DFGD algorithm.}}
		\label{figure4}
	\end{minipage}
\end{figure}

\subsection{Distributed functional optimization over reproducing kernel Hilbert spaces}

In statistical learning and machine learning, kernel-based learning occupies a crucial position. The typical problem model in kernel-based learning theory is to minimize the expected risk of the prediction function defined by
\bes
\mb R(f)=\int_{X\times Y}\huaL(f(x)-y)d\rho(x,y)
\ees
in a reproducing kernel Hilbert space $H_K$ induced by a Mercer kernel $K$ (namely, continuous, symmetric and positive semi-definite). Here,  $\huaL$ represents an appropriately smooth loss function and $\rho$ is an unknown probability distribution supported on $X\times Y$ that governs the sampling process. However, in most cases, the measure $\rho$ is unknown, accordingly, the corresponding regression function that minimizes the above functional $\mb R$ is also unknown. Hence we are only able to realize the learning task via training samples. That is, to optimize the functional optimization problem related to the empirical risk defined by
\bes
\min_{f\in H_K}\wh {\mb R}(f)=\sum_{i=1}^m\wh {\mb R}_i(f)
\ees
where the local functional $\mb{\wh R}_i: H_K\rightarrow\mbb R$ is the empirical risk associated with the agent $i\in\huaV$ defined by
\bes
\wh  {\mb R}_i(f)=\f{1}{n_i}\sum_{s=1}^{n_i}\huaL(f(x_{i,s}),y_{i,s})+\f{\lambda_i}{2}\|f\|_{H_K}^2.
\ees
Here,  $\f{\lambda_i}{2}\|f\|_{H_K}^2$ represents an RKHS regularization term with appropriate scaling parameter $\lambda_i>0$. $D_i=\{(x_{i,s},y_{i,s})\}_{s=1}^{n_i}\subseteq X\times Y$ is the sample data set with  cardinality $n_i$ drawn according to the unknown probability distribution $\rho$ and only observable by their associated local agent $i\in\huaV$ ($i$ only has access to $D_i$). Inspired by the theory of this work, we know that a fully decentralized kernel-based DFGD (KB-DFGD) in RKHS can be   designed as\\

\noi \textbf{KB-DFGD:}
\bes
\left\{
\begin{aligned}
	h_{i,t+1}=&f_{i,t}-\eta\left[\f{1}{n_i}\sum_{s=1}^{n_i}\huaL'(f_{i,t}(x_{i,s}),y_{i,s})K_{x_{i,s}}+\lambda_i f_{i,t}\right],\\
	f_{i,t+1}=&\sum_{j=1}^m[P_t]_{ij}	h_{j,t+1}.
\end{aligned}
\right. 
\ees
Here, we have used the notation $K_x=K(x,\cdot)$ for a Mercer kernel. The reason for the structure of the first step of the above KB-DFGD is due to the Fr\'echet derivative of $\wh {\mb R}_i$ satisfies
\bes
\mb D_f \wh {\mb R}_i=\f{1}{n_i}\sum_{s=1}^{n_i}\huaL'(f(x_{i,s}),y_{i,s})K_{x_{i,s}}+\lambda_i f_{i,t}.
\ees
The loss function $\huaL$ here can be flexibly chosen according to the requirements of different learning tasks. The regularization parameter $\lambda_i=0$ corresponds to the regularization-free case. For examples, when handling standard least squares regression learning problems, the standard least squares loss is selected as
$\huaL(f(x),y)=(f(x)-y)^2$, according to our theory, a regularization-free least squares KB-DFGD (LS-KB-DFGD) can be given as\\

\noi \textbf{LS-KB-DFGD:}
\bes
\left\{
\begin{aligned}
	h_{i,t+1}=&f_{i,t}-\eta\left[\f{1}{n_i}\sum_{s=1}^{n_i}(f_{i,t}(x_{i,s})-y_{i,s})K_{x_{i,s}}\right],\\
	f_{i,t+1}=&\sum_{j=1}^m[P_t]_{ij}	h_{j,t+1}.
\end{aligned}
\right. 
\ees

 In order to handle some robust learning tasks (see e.g. \cite{ghs2018}), $\huaL$ can be selected as a robust loss form in terms of $\huaL_\sigma(f(x),y)=\sigma^2W(\f{(f(x)-y)^2}{\sigma^2})$ with $W:\mbb R_+\rightarrow\mbb R$ being appropriately selected windowing function and $\sigma>0$ being some robustness parameter. For example, the Welsch loss $\huaL_\sigma(u)=\sss^2[1-\exp(-\f{u^2}{2\sss^2})]$ (nonconvex), the Cauchy loss $\huaL_\sigma(u)=\sss^2[\log(1+\f{u^2}{2\sss^2})]$ (nonconvex) and the Fair loss $\huaL_\sss(u)=\sss^2[\f{|u|}{\sss}-\log(1+\f{|u|}{\sss})]$ (see e.g. \cite{ghs2018}). Accordingly,  our KB-DFGD algorithm reduces to the following robust kernel-based DFGD form\\ 
 
\noi \textbf{Robust KB-DFGD:}
\bes
\left\{
\begin{aligned}
	h_{i,t+1}=&f_{i,t}-\f{\eta}{n_i}\sum_{s=1}^{n_i}W'(\theta_{i,t,\sss}(x_{i,s},y_{i,s}))(f_{i,t}(x_{i,s})-y_{i,s})K_{x_{i,s}},\\
	f_{i,t+1}=&\sum_{j=1}^m[P_t]_{ij}	h_{j,t+1}.
\end{aligned}
\right. 
\ees
with $\theta_{i,t,\sss}(x,y)=\f{(f_{i,t}(x)-y)^2}{\sss^2}$. The unified theoretical framework established in this work has indicated that, under mild conditions, LS-KB-DFGD and robust KB-DFGD possess nice convergence performance. Although the convergence theory of these algorithms has not been established separately in existing work, as a corollary of the theory in this paper, good convergence performances of these algorithms have been ensured under mild conditions.

 Till now, we have already provided several new decentralized kernel learning approaches via the time-varying multi-agent system based on the theory developed in this paper. In practical applications related to kernel-based learning, to facilitate the actual operation and adjustment of algorithms, we can apply methods such as the representer theorem \cite{shs2001}, random feature techniques \cite{sd2016} or Nystr\"om approximation approaches \cite{fls2025} to transition to algorithmic forms that are easier to handle in real-world scenarios.

\section{Numerical experiments}



In this section, we conduct numerical experiments to illustrate various features of the distributed convex and nonconvex functional optimization over Banach spaces, with examples on the distributed kernel-based functional optimization within reproducing kernel Hilbert spaces. Our first experiment investigates the LS-KB-DFGD algorithm, which employs the standard least squares loss:
\begin{eqnarray}
\min_{f\in H_K} \mb J(f)= \frac{1}{m} \sum_{i=1}^m  \frac{1}{n_i} \sum_{s=1}^{n_i} \left(f(x_{i,s})-y_{i,s}\right)^2, 
\end{eqnarray}
where $x_{i,s} \in \mathbb R^d$ and $y_{i,s} \in \mathbb R$ are data known only to the agent $i$. We adopt the following configurations in our simulation examples. The Mercer kernel $K$ is chosen as the Gaussian kernel $K(x,y)= \exp\left(-\frac{\|x-y\|_2^2}{\gamma^2}\right)$, $x,y \in \mathbb{R}^d$. The parameters in LS-KB-DFGD are chosen as follows: we consider the number of agents $m=30$, the number of samples in each agent $n_1= n_2 = \cdots = n_m = 10$, the input dimension $d=10$, the bandwidth of the Gaussian kernel $\gamma=0.33$. Each element of the input vector $x_{i,s}$ is generated from a uniform distribution over the interval $[-1,1]$, and the response is generated by
$y_{i,s} = \langle a, x_{i,s} \rangle + \epsilon$,
where $[a]_i=1$ for $1 \leq i \leq \lfloor d/2 \rfloor$ and $0$ otherwise, and the noise $\epsilon$ is drawn i.i.d. from a normal distribution $\mathcal N(0,1)$. The initial functions of each agent are chosen as $f_{i,1}=0$ for $i \in \{1,2,\dots,m\}$. The stepsize $\eta$ is selected as $\eta=\f{1}{\sqrt{T}}$. In our experiments, we evaluate the optimization error  $\mb J(\ww f_{\ell,T})-\mb J(f^*)$ ($\ell \in \huaV$), where $\ww f_{\ell,T}=\f{1}{T}\sum_{t=1}^Tf_{\ell,t}$. 

We illustrate the convergence behavior of the LS-KB-DFGD algorithm by plotting the maximum and minimum of the optimization errors $\{\mb J(\ww f_{\ell,T})-\mb J(f^*)\}_{\ell=1}^{m}$ across $m$ agents, against the total number of iterations $T$. The results are depicted in Fig. \ref{figure1}, demonstrating that all agents in the LS-KB-DFGD algorithm converge at a consistent rate.
Next, we examine how the network size (number of agents $m$) affects the convergence of the LS-KB-DFGD algorithm. We analyze three distinct numbers of agents in a ring network: $m = 30$, $m=40$, and $m = 50$, and present the graphs of the maximum of the optimization errors across all agents versus the total number of iterations $T$. The results illustrated in Fig. \ref{figure2} indicate that the LS-KB-DFGD algorithm is more effective with a smaller network size.


In our second experiment, we examine the Robust KB-DFGD algorithm, where the nonconvex Cauchy loss $\huaL_\sigma(u)=\sss^2[\log(1+\f{u^2}{2\sss^2})]$ is considered:
\begin{eqnarray}
\min_{f\in H_K} \mb J(f) = \frac{1}{m} \sum_{i=1}^m   \frac{1}{n_i} \sum_{s=1}^{n_i} \huaL_\sigma\left(f(x_{i,s})-y_{i,s}\right).
\end{eqnarray}
The configurations remain consistent with the previous experiments, with the exception that we set the scale parameter $\sigma=1$. Moreover, we introduce two outliers for the first two data in each two agents, specifically $\tilde{y}_{i,1}= y_{i,1}+5$ and $\tilde{y}_{i,2}= y_{i,2}-5$, for $i \in \{1,\dots, m\}$.
We demonstrate the convergence characteristics of the Robust KB-DFGD algorithm by plotting the maximum and minimum of the optimization errors across $m$ agents against the total number of iterations $T$. The results are displayed in Fig. \ref{figure4}, indicating that all agents within the Robust KB-DFGD algorithm converge at a comparable rate. 

\section{Conclusion}
In this paper, we systematically present a new distributed functional optimization (DFO) framework  in Banach spaces, based on time-varying multi-agent systems. We have studied both convex and nonconvex DFO problems. For these two categories, we have established comprehensive convergence theories for DFMD in Banach spaces and DFGD in Hilbert spaces, yielding  satisfactory convergence rates. These include a convergence rate of $\huaO(\f{1}{\sqrt{T}})$ for the ergodic sequence DFMD in Banach spaces under the convex setting, an average rate of $\huaO(\f{1}{\sqrt{T}})$ for the DFGD in Hilbert spaces under the nonconvex setting, and an $R$-linear convergence rate up to a solution level that is proportional to stepsize for only the last iteration of any local state variable. Furthermore, we demonstrates the broad applicability of our theory in machine learning, statistical learning such as MS-DFMD on Radon measure spaces and KS-DFGD, LS-KB-DFGD, robust KB-DFGD on RKHSs. Finally, experiments are conducted to show the good performances of the proposed algorithms.

\vspace{-14 mm}

\end{document}